\PassOptionsToPackage{colorlinks=true,linkcolor=blue,citecolor=blue,urlcolor=blue,bookmarks=true,pdfencoding=unicode,hypertexnames=false}{hyperref}
\documentclass[pdflatex,sn-mathphys-num]{sn-jnl}

\usepackage[utf8]{inputenc}
\usepackage[T1]{fontenc}
\usepackage{lmodern}
\usepackage{microtype}
\usepackage{mathtools,amssymb,amsfonts,amsthm}
\usepackage{enumitem}
\usepackage{graphicx}
\usepackage{float}
\usepackage{url}
\usepackage{tikz}
\usetikzlibrary{arrows.meta, positioning, calc}
\usepackage{tikz-cd}
\usepackage{booktabs}
\usepackage{silence}
\newcommand{\Qp}{\mathbb{Q}_p}
\newcommand{\Zp}{\mathbb{Z}_p}
\newcommand{\Cp}{\mathbb{C}_p}
\newcommand{\Fp}{\mathbb{F}_p}

\newcommand{\PP}{\mathbb{P}}

\newcommand{\norm}[1]{\left\|#1\right\|}
\newcommand{\Ball}[2]{B(#1,#2)}
\newcommand{\cO}{\mathcal{O}}
\newcommand{\Int}{\mathrm{Int}}

\newcommand{\Res}{\mathrm{Res}}

\newcommand{\End}{\mathrm{End}}
\newcommand{\Teich}{\mathrm{Teich}}

\newcommand{\Z}{\mathbb{Z}}
\newcommand{\F}{\mathbb{F}}
\newcommand{\Gal}{\mathrm{Gal}}

\theoremstyle{thmstyleone}
\newtheorem{theorem}{Theorem}[section]
\newtheorem{proposition}[theorem]{Proposition}
\newtheorem{lemma}[theorem]{Lemma}
\newtheorem{corollary}[theorem]{Corollary}
\theoremstyle{thmstyletwo}
\newtheorem{example}[theorem]{Example}
\newtheorem{remark}[theorem]{Remark}
\newtheorem{observation}[theorem]{Observation}
\theoremstyle{thmstylethree}
\newtheorem{definition}[theorem]{Definition}



\title[Rational Interpreters for Discrete Dynamics]{Rational Interpreters for Discrete Dynamics: Existence, Exactness, and Decomposition over \texorpdfstring{$p$-adic}{p-adic} Fields}

\author*[1]{\fnm{J. Rogelio} \sur{P\'erez-Buend\'ia} (ORCID 0000-0002-7739-4779)}\email{rogelio.perez@cimat.mx}

\affil*[1]{\orgdiv{SECIHTI}, \orgname{Centro de Investigaci\'on en Matem\'aticas (CIMAT), Unidad M\'erida}, \orgaddress{\city{M\'erida}, \state{Yucat\'an}, \country{M\'exico}}}

\pacs[MSC 2020]{37P20 (Primary), 11S80, 13K05, 14G22, 37B10 (Secondary)}

\keywords{non-Archimedean dynamics; p-adic dynamics; discrete dynamical systems; functional graphs; Witt vectors; rigid analytic geometry; Runge approximation; good reduction; Chinese remainder theorem; congruence-preserving maps; inverse lifting problem; ball dynamics}

\abstract{We address an inverse problem in non-Archimedean dynamics: given a finite discrete dynamical system (equivalently, a functional graph on $N$ states), construct a continuous $p$-adic dynamical system whose residue-level behavior reproduces the prescribed transitions. Using the cylinder partition of $\cO_K$ (viewed as \emph{Witt cylinders} for unramified $K/\mathbb{Q}_p$), we encode states by pairwise disjoint closed balls and formalize an \textbf{interpreter} as a map sending each state ball into its target ball.

Our main existence result constructs rational interpreters that are analytic (hence pole-free) on the prescribed state cylinders, combining rigid-analytic Runge approximation with finite interpolation constraints. Under a linear-dominance condition on each cylinder, ball images are explicit and locally affine, leading to a robust classification of discrete behavior into contractive, indifferent, and expansive regimes. Good reduction provides a selection principle for natural interpreters; effective degree and height bounds for general rational interpreters remain open.

For composite alphabets we prove a \textbf{Dynamic Chinese Remainder Theorem} for congruence-preserving systems: the CRT isomorphism $\Theta:\mathbb{Z}/m\mathbb{Z}\xrightarrow{\sim}\prod_i\mathbb{Z}/p_i^{k_i}\mathbb{Z}$ (for $m=\prod p_i^{k_i}$) yields a factorization of the \emph{dynamics} (equivalently, the functional graph) on $\mathbb{Z}/m\mathbb{Z}$ into dynamics on the prime-power components, compatible with reduction. Finally, we discuss an inverse-limit (profinite) extension: compatible towers define a $1$-Lipschitz map on $\mathbb{Z}_p$, while selecting compatible analytic/rational interpreters across levels becomes a separate problem.
}

\begin{document}
\raggedbottom

\maketitle

\tableofcontents

\section{Introduction}
\label{sec:introduction}

\subsection{Discrete Dynamics and the Lifting Problem}

Functional graphs---directed graphs where every vertex has out-degree one---are the standard combinatorial representation of finite discrete dynamical systems (DDS)~\cite{JarrahLaubenbacher2007}. They model cellular automata, regulatory networks, and orbits of rational maps over finite fields, and sit at the interface of discrete dynamics, arithmetic dynamics, and $p$-adic analysis. The \emph{lifting problem} is this: given a finite discrete system, can one build a continuous dynamical system over a local field whose ``shadow'' at a chosen resolution is exactly that discrete system? In algebraic terms: given a discrete map $F$ on a finite ring $R$ (e.g.\ $R = \mathbb{Z}/p^k\mathbb{Z}$), find a continuous map $\phi$ on the valuation ring such that the depth-$k$ reduction $\pi^{(k)}$ satisfies $\pi^{(k)} \circ \phi = F \circ \pi^{(k)}$. Under the ball-semantics notion (interpretation with inclusion), existence holds over $K=\mathbb{C}_p$ (Theorem~\ref{thm:analytic_existence}, Theorem~\ref{thm:runge}); the construction is given by a piecewise affine interpreter followed by Runge approximation. The Runge step does not furnish explicit degree or height bounds; see Section~\ref{sec:discussion} and Remark~\ref{rem:degree_bound_limits}. Exact ball mapping is an open property under a linear dominance condition.

\begin{sloppypar}
The ring of integers $\cO_K$ of a finite unramified extension $K/\Qp$ (with residue field $\F_q$), identified with Witt vectors $\cO_K \cong W(\F_q)$ (via the Witt isomorphism in the unramified case), supplies the \emph{metric and hierarchical structure} that polynomial interpolation over $\F_q$ alone cannot, and it aligns with multi-valued logic and $p$-adic models in biology~\cite{Thomas1973, Dragovich2007, Khrennikov2009, YurovaAxelssonKhrennikov2024codes, YurovaAxelssonKhrennikov2024universal}.

\begin{quote}
\begin{sloppypar}
A discrete configuration is encoded as a \emph{Witt cylinder} (a $p$-adic ball) in $\cO_K$. A discrete update rule $F$ is interpreted by a global rational or analytic map $\phi$ that sends cylinders to cylinders. The hierarchy of Witt vectors ensures that dynamics are compatible across resolution levels: refining or coarsening the discrete grid corresponds to changing the radius of the balls.\allowbreak
\end{sloppypar}
\end{quote}
We use the term \emph{Witt cylinder} for the cylinder balls in $\cO_K \simeq W(\kappa)$ (for unramified $K$), i.e., a depth-$n$ congruence ball in $\cO_K$ (residue classes modulo $p^n$ viewed as closed balls of radius $p^{-n}$).
\end{sloppypar}

Viewing the discrete system as ball-to-ball dynamics brings in the full machinery of $p$-adic analysis---multipliers, attractors, and stability theory. Combinatorially identical states then acquire distinct analytical signatures: a contractive state is robust to perturbations, an expansive one is sensitive. The continuous lift thus adds a quantitative layer to the discrete graph.

\subsection{Scientific Context: From Berkovich Space to Reduction Theory}

Non-Archimedean dynamics has been transformed by the Berkovich projective line~\cite{Berkovich1990} and Fatou-Julia theory~\cite{RiveraLetelier2003}. The literature has focused largely on the \emph{forward problem}: given $\phi \in K(z)$, describe its action on residue rings (e.g., Jones-Manes~\cite{JonesManes2014}, Hsiao~\cite{Hsiao2013}). We address the \emph{inverse problem}: given a finite functional graph $G$, construct a rational map $\phi$ whose ball dynamics realizes $G$ at a chosen resolution. This inverse lifting problem ties in with reduction theory: Benedetto~\cite{BenedettoReduction} showed that maps with good reduction behave uniformly on the residue field; P\'erez-Buend\'ia~\cite{PerezBuendia2025Arboreal} extended this to arboreal representations. In our setting, \textbf{interpretation with inclusion} (ball images contained in targets) is closely tied to good reduction: strict good reduction (degree-preserving reduction) prevents the continuous map from ``collapsing'' discrete states and offers a selection principle when many interpreters exist (Section~\ref{subsec:reduction}).\footnote{When we invoke reduction as a selection principle, we work with strict good reduction; see Section~\ref{subsec:reduction} and Example~\ref{ex:non_strict_good_reduction}.} The same setting yields \emph{profinite coherence}: a discrete system can be viewed as a compatible tower of maps, and the inverse limit induces global continuous dynamics on the unramified base $\cO_{K_0}$ (viewed inside $\cO_{\mathbb{C}_p}$ via Remark~\ref{rem:field_base}), in line with $p$-adic biological models~\cite{AnashinKhrennikov} and profinite dynamics~\cite{Poonen2014}.

\subsection{Main Contributions}

The inverse lifting problem is solvable and the solution is decomposable; exact cylinder mapping is stable under linear dominance (Section~\ref{sec:robust}). The main contributions are the following.
\begin{itemize}[leftmargin=2em]
\item \textbf{Witt Vector Semantics:} Multi-state alphabets are encoded via unramified $p$-adic bases (finite-level truncations of $\cO_{K_0}\cong W(\overline{\F}_p)$), with discrete states identified to Witt cylinders.
\item \textbf{Interpretation:} The base notion is that of an \emph{interpreter} (ball images meet targets), with three types per ball (contractive, indifferent, expansive) and global conditions (interpretation with inclusion, exact interpretation) phrased as congruence relations in the Witt ring.
\item \textbf{Existence via Runge Approximation:} Density of rational functions on affinoid domains (rigid Runge theorem) yields, for any finite functional graph, a rational interpreter on a finite union of disjoint balls; the construction is piecewise affine interpreter then Runge approximation (non-effective bounds).
\item \textbf{Robust Exactness:} A ``linear dominance'' criterion ensures that exact cylinder mapping is an open, stable property in the space of analytic maps.
\item \textbf{Dynamic CRT and global product synthesis:} For composite alphabets, the interpretation lifts to a product of local domains. Proposition~\ref{prop:crt_product} establishes that the continuous phase space for a system of size $m = \prod p_i^{k_i}$ is the product space $\prod_i \cO_i$ (each $\cO_i \cong \mathbb{Z}_{p_i}$ or an unramified extension); the resulting horizontal vs.\ vertical decomposition is discussed in Section~\ref{subsec:horizontal_vertical}.
\item \textbf{Reduction Principle:} Good reduction provides a selection principle for preferred interpreters when many exist.
\end{itemize}

\subsection{Arithmetic information in functional graphs}
\label{sec:arith_graphs_lifting}
Functional graphs over residue rings/fields encode solvability patterns of polynomial and rational equations.
The interpreter constraints on balls/cylinders
force a prescribed finite-level solvability/branching pattern, while vertical refinement asks for compatibility across depths.

\paragraph{Indegree as number of solutions.}\label{par:indegree_solutions}
Let $q=p^f$ and let $f:\F_q\to\F_q$ be a function (e.g.\ induced by a polynomial in $\F_q[x]$).
In the functional graph of $f$, the indegree of a node $a\in\F_q$ is
\[
\#f^{-1}(a)=\#\{x\in\F_q:\ f(x)=a\},
\]
so leaves (nodes of indegree $0$) are precisely the values not attained by $f$; equivalently, they measure the failure of surjectivity of $f$ over $\F_q$, i.e.\ the obstructions to solvability of $f(x)=a$ in $\F_q$.

\paragraph{Example: $x\mapsto x^2$ and quadratic residues.}\label{ex:square_map_fp}
For odd $p$ and $f(x)=x^2$ on $\F_p$, the leaves are exactly the non-quadratic residues.
Equivalently, each $a\in\F_p^\times$ has indegree $2$ if $a$ is a square and indegree $0$ otherwise, while $0$ has indegree $1$.

\paragraph{Example: $x\mapsto x^2+1$ and extension of scalars.}\label{ex:x2_plus1_extension}
For odd $p$ and $f(x)=x^2+1$, solvability of $f(x)=0$ is solvability of $x^2=-1$, which depends on $p\bmod 4$:
if $p\equiv 1\ (\mathrm{mod}\ 4)$ then $0$ has indegree $2$ in $\F_p$, whereas if $p\equiv 3\ (\mathrm{mod}\ 4)$ then $0$ is a leaf in $\F_p$
but becomes non-leaf after passing to $\F_{p^2}$.
More generally, enlarging the residue field can attach new preimages to previously-leaf nodes, reflecting refinement of solvability
data under unramified extension.

\paragraph{Bridge to Witt cylinders and lifting.}\label{par:bridge_witt_lifting}
Finite alphabets and congruence classes are modeled by Witt cylinders in the unramified base $\cO_{K_0}\cong W(\overline{\F}_p)$
(see Remark~\ref{rem:field_base}). A discrete transition system over $\Z/p^k\Z$ is then finite-resolution data, and vertical refinement
(dynamical Hensel-type lifting) asks when prescribed cycles/branches persist across depths.
Existence over $\Cp$ produces analytic/rational maps whose behavior is constrained on these cylinders, thereby linking finite arithmetic dynamics
to non-Archimedean analytic dynamics.

\paragraph{Horizontal vs.\ vertical refinement.}\label{par:horizontal_vs_vertical_refinement}
There are two distinct arithmetic refinements of the branching/solvability data encoded by a functional graph.
\emph{Horizontal refinement} enlarges the residue field $\kappa=\F_p$ to $\kappa'=\F_{p^f}$ (unramified extension), which can create new preimages because equations may acquire solutions after extension of scalars (cf.\ Example~\ref{ex:x2_plus1_extension}).
\emph{Vertical refinement} increases the depth $n$ in the truncation tower $W_{n+1}(\kappa)\to W_n(\kappa)$ (lifting mod $p^n\to p^{n+1}$), asking for Henselian persistence of prescribed branches/cycles.
These two operations are logically independent: residue-field extension corresponds to changing the Witt base $W(\kappa)$, and depth refinement to refining cylinders via truncation.

\subsection{Organization}

Notation and preliminaries are in Sections~\ref{sec:encoding}--\ref{sec:prelim}. The analytic core---semantics (Section~\ref{sec:semantics}), existence via Runge (Section~\ref{sec:existence}), and robust exactness (Section~\ref{sec:robust})---leads to the arithmetic extensions in Section~\ref{sec:composite_alphabets}: Witt vectors, Dynamic CRT, horizontal vs.\ vertical decomposition (Section~\ref{subsec:horizontal_vertical}), and pro-$p$ towers with dynamical Hensel. Section~\ref{sec:closing_the_loop} treats reduction and selection criteria; Section~\ref{sec:examples} gives worked examples. Section~\ref{sec:discussion} discusses relations to other work, open problems, and the passage to radius-zero (profinite) interpretation (Appendix~\ref{app:profinite_addendum}); see also the subsection on arithmetic information in Section~\ref{sec:arith_graphs_lifting}.
 \section{Functional Graphs and \texorpdfstring{$p$}{p}-adic Hierarchical Encoding}
\label{sec:encoding}

Functional graphs---directed graphs $G_F = (X, E)$ with $E = \{(x, F(x)) : x \in X\}$---are the combinatorial representation of finite discrete dynamical systems $F: X \to X$. Every vertex has exactly one outgoing edge; in particular, \textbf{fixed points} (vertices $x$ with $F(x)=x$) are represented by a \textbf{self-loop} (the edge $(x,x)$). We adopt this convention throughout the paper and in all figures. In multiscale models, $X$ often has a product structure $X = \Sigma^n$, where $\Sigma = \{0, \dots, p-1\}$.

Fixing a coordinate order, we identify $x = (x_1, \dots, x_n) \in \Sigma^n$ with a $p$-adic integer via the encoding map:
\begin{equation}
\label{eq:encoding}
\mathrm{enc}_n(x) = \sum_{i=1}^n x_i p^{i-1}.
\end{equation}
\begin{sloppypar}
Each configuration is associated with a \emph{state cylinder} $B_x := \Ball{\mathrm{enc}_n(x)}{p^{-n}} \subset \Cp$.
\end{sloppypar}

\begin{lemma}[Cylinder Partition and Nesting]
\label{lem:cylinder_partition}
\begin{enumerate}[label=(\roman*)]
    \item The balls $\{B_x\}_{x \in \Sigma^n}$ form a partition of $\mathbb{Z}_p$ into $p^n$ pairwise disjoint closed balls of radius $p^{-n}$.
    \item For $m < n$, the truncation map $\pi_{n \to m}: \Sigma^n \to \Sigma^m$ (retaining the first $m$ $p$-adic digits) satisfies $B_x \subseteq B_{\pi_{n \to m}(x)}$.
\end{enumerate}
\end{lemma}

The radius $p^{-n}$ acts as a \emph{resolution parameter}. The state cylinders $B_x$ are the fibers of the depth-$n$ reduction $\pi^{(n)}: \Zp \to \mathbb{Z}/p^n\mathbb{Z}$; the equivalence between interpretation with inclusion and commutativity of this reduction is given in Section~\ref{sec:composite_alphabets} (Lemma~\ref{lem:cylinders_as_fibers}). Witt vectors and the Dynamic CRT are developed in that section.
 \section{Ultrametric Preliminaries}
\label{sec:prelim}

\subsection{\texorpdfstring{$p$}{p}-adic Balls and Basic Geometry}
Let $K$ be a complete algebraically closed non-Archimedean field (for instance, $K = \Cp$) with absolute value $|\cdot|$ normalized such that $|p| = p^{-1}$. For $a \in K$ and $r > 0$, we denote the closed ball by $\Ball{a}{r} = \{z \in K : |z-a| \leq r\}$. The ultrametric nature of the field implies the following fundamental nesting property.

\begin{lemma}[Nesting Property]
\label{lem:nesting}
If two balls $\Ball{a}{r}$ and $\Ball{b}{s}$ have a non-empty intersection, then one must be contained in the other.
\end{lemma}
\begin{proof}
This is a standard consequence of the ultrametric triangle inequality (see, e.g.,~\cite{BGR} or~\cite{FvdP}). If $|a-b| \leq \max(r,s)$, then the ball with the larger radius contains the other.
\end{proof}

\begin{remark}[Nesting and Profinite Towers]
\label{rem:nesting_profinite}
\begin{sloppypar}
The nesting property (Lemma~\ref{lem:nesting}) is fundamental to the hierarchical structure. In Section~\ref{sec:encoding}, we use balls of radius $p^{-n}$ to encode discrete configurations. The nesting structure $\{B_x\}_{x \in \Sigma^n}$ naturally forms an inverse system: for $m < n$, each ball $B_x$ (radius $p^{-n}$) is contained in a unique ball $B_{\pi_{n \to m}(x)}$ (radius $p^{-m}$).

This hierarchical structure is the geometric foundation for the profinite coherence developed in Section~\ref{sec:composite_alphabets}, where compatible towers of discrete maps induce global continuous dynamics on the unramified base $\mathcal{O}_{K_0} \cong W(\kappa)$ (see Remark~\ref{rem:field_base}).
\end{sloppypar}
\end{remark}

\subsection{Affine Maps and Ball Images}
The behavior of balls under analytic maps is best understood by starting with affine models.
\begin{lemma}[Affine Images]
\label{lem:affine_ball}
Let $\psi(z) = \beta + u(z-\alpha)$ be an affine map with $u \neq 0$. Then $\psi(\Ball{\alpha}{r}) = \Ball{\beta}{|u|r}$.
\end{lemma}
\begin{proof}
For any $z \in \Ball{\alpha}{r}$, we have $|\psi(z) - \beta| = |u(z-\alpha)| = |u||z-\alpha| \leq |u|r$, hence $\psi(\Ball{\alpha}{r}) \subseteq \Ball{\beta}{|u|r}$. Conversely, for any $w \in \Ball{\beta}{|u|r}$, the point $z = \alpha + u^{-1}(w-\beta)$ satisfies $|z-\alpha| = |u|^{-1}|w-\beta| \leq r$, so $z \in \Ball{\alpha}{r}$ and $\psi(z) = w$.
\end{proof}

\subsection{Non-Archimedean Cauchy Estimates}
The control of Taylor coefficients is essential for our results on robust exactness. Let $f \in \cO(\Ball{a}{r})$ be an analytic function with expansion $f(z) = \sum_{k \ge 0} c_k (z-a)^k$.

\begin{lemma}[Cauchy Bound]
\label{lem:cauchy}
For all $k \ge 0$, $|c_k| r^k \leq \norm{f}_{\Ball{a}{r}}$, where $\norm{f}$ denotes the supremum norm on the ball.
\end{lemma}
\begin{proof}
This is the non-Archimedean Cauchy estimate central to rigid analytic geometry and to $p$-adic dynamics (see Proposition~3.17 and Corollary~3.21 in~\cite{BenedettoDynamics}, and~\cite{AnashinKhrennikov}; for the rigid setting see also~\cite{BGR}).\footnote{The Archimedean analogue is the Cauchy inequality for holomorphic functions on a disk.} The bound follows from the fact that $|c_k| = |\Res_{z=a} (z-a)^{-k-1} f(z)| \leq r^{-k} \norm{f}_{\Ball{a}{r}}$ by the ultrametric property.
\end{proof}

\begin{lemma}[Coefficient Stability]
\label{lem:cauchy_diff}
If $f, g$ are analytic functions on $\Ball{a}{r}$ such that $\sup_{\Ball{a}{r}} |f-g| \leq \varepsilon$, then for each $k \ge 0$, we have
\[
|c_k(f) - c_k(g)| r^k \leq \varepsilon.
\]
In particular, $|f'(a) - g'(a)| r \leq \varepsilon$ and $|f(a) - g(a)| \leq \varepsilon$.
\end{lemma}
\begin{proof}
This follows by applying Lemma~\ref{lem:cauchy} to the difference $h = f-g$, noting that the coefficients of $h$ are exactly $c_k(f) - c_k(g)$.
\end{proof}
 \section{Semantics of Interpretation}
\label{sec:semantics}

\subsection{The General Setting}
The notion of a continuous map interpreting a discrete transition rule is formalized here through ball dynamics.

\begin{definition}[Ball System]
A \emph{ball system} is a triple $(\mathcal{B}, \tau, \mathcal{B}')$ where:
\begin{enumerate}
    \item $\mathcal{B} = \{B_i = \Ball{a_i}{r_i}\}_{i=1}^N$ is a finite collection of pairwise disjoint closed balls in $\Cp$.
    \item $\tau: \{1, \dots, N\} \to \{1, \dots, N\}$ is a discrete transition rule.
    \item $\mathcal{B}' = \{B_i' = \Ball{b_i}{t_i}\}_{i=1}^N$ is a family of target closed balls.
\end{enumerate}
\end{definition}

In the context of the hierarchical encoding discussed in Section~\ref{sec:encoding}, $B_i$ are the cylinders $B_x$ and $\tau$ is the update rule $F$ on the configurations.

\subsection{Interpretation and Three Interpretation Types}
A map is an \emph{interpreter} when its ball images meet the prescribed targets (non-empty intersection). The ultrametric nesting property forces a unique containment relation at each ball, which yields three interpretation types.

\begin{definition}[Interpreter]
\label{def:interpreter}
Let $(\mathcal{B}, \tau, \mathcal{B}')$ be a ball system. A map $\phi$ that is analytic in a neighborhood of $U = \bigcup B_i$ \emph{interprets} $\tau$ if for every $i$,
\[
\phi(B_i) \cap B'_{\tau(i)} \neq \varnothing.
\]
In particular, if $\phi$ is rational, then ``analytic in a neighborhood of $U$'' means that
$\phi$ has no poles on some open neighborhood of $U$ (hence no poles on $U$).
We say that $\phi$ has, at ball $i$,
\begin{enumerate}[label=(\roman*)]
    \item \textbf{contractive interpretation} if $\phi(B_i) \subsetneq B'_{\tau(i)}$ (image strictly inside target);
    \item \textbf{indifferent interpretation} if $\phi(B_i) = B'_{\tau(i)}$ (image equals target);
    \item \textbf{expansive interpretation} if $B'_{\tau(i)} \subsetneq \phi(B_i)$ (target strictly inside image).
\end{enumerate}
\end{definition}

\begin{remark}[Image structure under linear dominance]
\label{rem:image_ball}
For the interpreters constructed in Sections~\ref{sec:existence}--\ref{sec:robust}, $\phi$ is analytic without poles on each $B_i$ and satisfies \emph{linear dominance} (Definition~\ref{def:dominance}), so $\phi(B_i)$ is exactly a ball. Linear dominance implies local invertibility (the term $|\phi'(a_i)|$ dominates higher-order terms); by the ultrametric inverse function theorem (see, e.g., \cite{BGR} or \cite{BenedettoDynamics}),\footnote{The Archimedean analogue is the classical inverse function theorem.} $\phi(B_i) = \Ball{\phi(a_i)}{|\phi'(a_i)| r_i}$. Lemma~\ref{lem:nesting} then gives that for each $i$ one of $\phi(B_i)$ and $B'_{\tau(i)}$ contains the other, so the three interpretation types above apply.
Equivalently, if $\phi(B_i)=\Ball{\phi(a_i)}{s_i}$ with $s_i=|\phi'(a_i)|\,r_i$, then the three
types at $i$ are determined by the ratio
\[
\sigma_i:=\frac{s_i}{t_{\tau(i)}}=\frac{|\phi'(a_i)|\,r_i}{t_{\tau(i)}}:\quad
\begin{aligned}[t]
\sigma_i&<1 \Longleftrightarrow \text{contractive},\\
\sigma_i&=1 \Longleftrightarrow \text{indifferent},\\
\sigma_i&>1 \Longleftrightarrow \text{expansive}.
\end{aligned}
\]
\end{remark}

\begin{remark}[Inclusion and exact interpretation]
\label{rem:strong_exact}
We say that $\phi$ \emph{interprets $\tau$ with inclusion} if $\phi$ interprets $\tau$ and at every $i$ the interpretation is contractive or indifferent, i.e.\ $\phi(B_i) \subseteq B'_{\tau(i)}$ for all $i$. We say that $\phi$ \emph{exactly interprets} $\tau$ if $\phi$ interprets $\tau$ and at every $i$ the interpretation is indifferent, i.e.\ $\phi(B_i) = B'_{\tau(i)}$ for all $i$. Thus: exact $\Rightarrow$ interpreter with inclusion $\Rightarrow$ interpreter; the converses do not hold in general.
\end{remark}

\begin{table}[ht]
\centering
\small
\begin{tabular}{@{}p{4.2cm}p{0.62\textwidth}@{}}
\toprule
\textbf{Notion} & \textbf{Condition} \\
\midrule
Interpreter & $\phi(B_i) \cap B'_{\tau(i)} \neq \varnothing$ for all $i$ \\
\addlinespace
Contractive (at $i$) & $\phi(B_i) \subsetneq B'_{\tau(i)}$ \\
Indifferent (at $i$) & $\phi(B_i) = B'_{\tau(i)}$ \\
Expansive (at $i$) & $B'_{\tau(i)} \subsetneq \phi(B_i)$ \\
\addlinespace
Interpreter with inclusion & $\phi(B_i) \subseteq B'_{\tau(i)}$ for all $i$ (contractive or indifferent everywhere) \\
Exact interpreter & $\phi(B_i) = B'_{\tau(i)}$ for all $i$ (indifferent everywhere) \\
\bottomrule
\end{tabular}
\caption{Interpretation: base notion and three types per ball; inclusion and exact as global conditions.}
\label{tab:semantics}
\end{table}

\paragraph{Key equivalence (semantics).}
We use the following terms consistently: \emph{interpretation with inclusion} means $\phi(B_i) \subseteq B'_{\tau(i)}$ for all $i$; \emph{exact interpreter} means $\phi(B_i) = B'_{\tau(i)}$ for all $i$. When the state cylinders are the fibers of a depth-$k$ quotient map $\pi^{(k)}$ (as in Section~\ref{sec:encoding} and Section~\ref{sec:composite_alphabets}), the geometric and algebraic conditions coincide: for all $x$,
\[
\phi(B_x) \subseteq B_{F(x)} \quad \Longleftrightarrow \quad \pi^{(k)} \circ \phi = F \circ \pi^{(k)} \quad \text{(see Lemma~\ref{lem:cylinders_as_fibers}).}
\]
This equivalence makes interpretation with inclusion checkable via the reduction diagram rather than by inspecting ball images directly.

\subsection{Obstructions to Interpretation}
The assumption of disjointness for the source balls is non-trivial. If balls overlap, the transition rule must satisfy severe compatibility constraints.
\begin{observation}
If $B_1 \subset B_2$, then for any function $\phi$, we must have $\phi(B_1) \subseteq \phi(B_2)$. Consequently, no interpreter can exist if the prescribed target balls $B_{\tau(1)}'$ and $B_{\tau(2)}'$ do not satisfy the same inclusion property.
\end{observation}
This highlights why the cylinder partition $\{B_x\}$ in $\mathbb{Z}_p$ is the domain we use to interpret finite discrete dynamics: at a fixed depth $n$ (radius $p^{-n}$), the partition contains $p^n$ disjoint balls, so any interpreted system at that depth must satisfy $|X| \le p^n$; conversely, any finite $|X| = N$ can be realized by choosing depth $n$ with $p^n \ge N$. The Witt and CRT development of these semantics (congruence relations, alphabets of size $p^k$, compatibility across resolution levels) is given in Section~\ref{sec:composite_alphabets}.
 \section{Existence of Rational Interpreters}
\label{sec:existence}

For the existence results in this section we work in $K = \mathbb{C}_p$ (algebraically closed, divisible value group); the arithmetic extensions in Section~\ref{sec:composite_alphabets} use the unramified base field $K_0$ with $\mathcal{O}_{K_0} \cong W(\kappa)$ (see Remark~\ref{rem:field_base}).

\subsection{Piecewise Affine Models}
The simplest construction of an interpreter is to define it ball-by-ball using affine maps. The following existence result is a reformulation, in our ball-system notation, of the local building block underlying the $\varepsilon$-gluing construction in P\'erez-Buend\'ia and Nopal-Coello~\cite{PBNC2025}. Theorem~4.2 there proves existence of a global rational $\phi$ that $\varepsilon$-approximates prescribed local analytic/rational data on finitely many disjoint balls, under the hypothesis that the image of the union of balls under each local map $f_i$ is contained in $B_1(0)$. In our application we choose local models $f_i$ (e.g.\ affine or locally dominant maps) that already encode the intended ball-to-ball behavior; then gluing yields a global $\phi$ with the required accuracy. We give an explicit statement and short proof in our setting for completeness. We first establish a standard fact about analytic functions on disconnected affinoid domains.

\begin{lemma}[Algebra of Functions on Disjoint Affinoids]
\label{lem:disjoint_affinoids}
Let $U = \bigsqcup_{i=1}^N U_i$ be a finite disjoint union of affinoid domains in $\PP^1(K)$. Then the algebra of analytic functions on $U$ decomposes as a product:
\[
\cO(U) \cong \prod_{i=1}^N \cO(U_i).
\]
In particular, giving analytic functions $f_i \in \cO(U_i)$ for each $i$ uniquely determines an analytic function $f \in \cO(U)$ such that $f|_{U_i} = f_i$.
\end{lemma}
\begin{proof}
This is a standard result in rigid analytic geometry, used throughout $p$-adic dynamics (see~\cite{BenedettoDynamics}; for the rigid setting see~\cite{FvdP} or~\cite{BGR}).\footnote{The Archimedean analogue is that holomorphic functions on a disjoint union of disks are determined by their restrictions.} Since the affinoids $U_i$ are pairwise disjoint, there exist disjoint admissible opens $V_i$ containing $U_i$ such that $U = \bigcup V_i$. The sheaf property of $\cO$ on admissible covers implies that $\cO(U) = \prod_i \cO(U_i)$.
\end{proof}

\begin{theorem}[Analytic Existence]
\label{thm:analytic_existence}
Let $(\mathcal{B}, \tau, \mathcal{B}')$ be a ball system, where $K$ is a complete algebraically closed non-Archimedean field with divisible value group (e.g., $K = \Cp$). Assume that all radii $r_i$ and $t_i$ belong to the value group $|K^\times|$. Then there exists an analytic function $\psi$ defined on the affinoid union $U = \bigcup B_i$ that exactly interprets $\tau$. In fact, $\psi$ can be chosen piecewise affine:
\[
\psi|_{B_i}(z) = b_{\tau(i)} + u_i(z - a_i), \quad \text{where } |u_i| = \frac{t_{\tau(i)}}{r_i}.
\]
\end{theorem}
\begin{proof}
Since $K$ is algebraically closed and its value group is divisible, for any $i$ we can find $u_i \in K$ such that $|u_i| = t_{\tau(i)}/r_i$. Define $\psi_i(z) = b_{\tau(i)} + u_i(z - a_i)$ on each ball $B_i$. By Lemma~\ref{lem:affine_ball}, we have $\psi_i(B_i) = B_{\tau(i)}'$. Since the balls $B_i$ are pairwise disjoint, Lemma~\ref{lem:disjoint_affinoids} ensures that the family $\{\psi_i\}$ uniquely determines an analytic function $\psi \in \cO(U)$ such that $\psi|_{B_i} = \psi_i$ for all $i$. This $\psi$ exactly interprets $\tau$.
\end{proof}

\begin{remark}[Radii in Unramified Extensions]
\label{rem:radii_unramified}
In unramified extensions $K_0/\mathbb{Q}_p$ (where the value group is $p^{\mathbb{Z}}$), the hypothesis that all radii belong to $|K_0^\times|$ restricts them to powers of $p$. This is consistent with our setting, as the state cylinders in Section~\ref{sec:encoding} are constructed with radii $p^{-n}$ for $n \in \mathbb{N}$. The algebraically closed case $K = \mathbb{C}_p$ (with value group $p^{\mathbb{Q}}$) allows more flexibility, but the construction remains valid as long as the radii are in the value group.
\end{remark}

\subsection{Global Rational Approximants (Runge)}
We work in $K = \mathbb{C}_p$ (algebraically closed) as in Remark~\ref{rem:field_base}. The piecewise affine interpreter $\psi$ constructed in Theorem~\ref{thm:analytic_existence} is analytic on the affinoid domain $U = \bigcup B_i$, but it is not globally defined on $\PP^1(K)$. To obtain a global interpreter (for instance, a rational function on $\PP^1$), we need to approximate $\psi$ by rational functions. The key tool is the \emph{Runge approximation theorem} in rigid analytic geometry, which states that rational functions are dense in the algebra of analytic functions on affinoid domains. For our setting, where $U$ is a finite union of pairwise disjoint closed balls (a Weierstrass domain), this density holds under the supremum norm topology. This allows us to approximate the piecewise affine $\psi$ arbitrarily closely by rational functions, while preserving the ball-dynamics semantics through the robust exactness results of Section~\ref{sec:robust}.

We use the following standard density result. A preliminary fact about polynomial interpolation at centers (used when discussing degree bounds) is recorded next.

\begin{proposition}[Polynomial interpolation at centers]
\label{prop:lagrange_centers}
Let $a_1,\dots,a_N\in K$ be pairwise distinct points and let $y_1,\dots,y_N\in K$.
Then there exists a unique polynomial $P\in K[z]$ with $\deg(P)\le N-1$ such that
$P(a_i)=y_i$ for all $i$.
\end{proposition}
\begin{proof}
This is the classical Lagrange interpolation theorem. Explicitly,
\[
P(z)=\sum_{i=1}^N y_i \prod_{j\ne i}\frac{z-a_j}{a_i-a_j}.
\]
\end{proof}

\begin{remark}[What this does \emph{not} imply]
\label{rem:degree_bound_limits}
Proposition~\ref{prop:lagrange_centers} controls the discrete update on the chosen representatives (centers),
but it does \emph{not} by itself ensure any ball-mapping property such as
$P(B_i)\subseteq B'_{\tau(i)}$ (or equality), nor does it provide uniform approximation
$\|P-\psi\|_U<\varepsilon$ for arbitrarily small $\varepsilon$ with degree bounded by $N-1$.
Uniform approximation on affinoids is obtained via Runge-type density (Theorem~\ref{thm:runge}),
typically without a degree bound independent of $\varepsilon$.
\end{remark}

\begin{theorem}[Rigid Runge Approximation]
\label{thm:runge}
Let $U \subset \PP^1(K)$ be a finite union of pairwise disjoint closed balls. Then the rational functions $K(z)$ are dense in $\cO(U)$ under the supremum norm (see~\cite{FvdP} or~\cite{BGR}). That is, for any $\psi \in \cO(U)$ and $\varepsilon > 0$, there exists $\phi \in K(z)$ such that $\sup_{z \in U} |\phi(z) - \psi(z)| < \varepsilon$.
\end{theorem}

\begin{lemma}[Runge approximation implies finite-point control]
\label{lem:runge_finite_constraints}
Let $U=\bigsqcup_{i=1}^N B_i\subset \PP^1(K)$ be a finite union of pairwise disjoint closed balls,
let $S=\{s_1,\dots,s_k\}\subset U$ be finite, and let $\psi\in\cO(U)$.
Given $\varepsilon>0$, there exists $\phi\in K(z)$ such that
$\norm{\phi-\psi}_U<\varepsilon$ and hence $|\phi(s_j)-\psi(s_j)|\le \varepsilon$ for all $j$.
\end{lemma}
\begin{proof}
By Theorem~\ref{thm:runge} there exists $\phi\in K(z)$ with $\norm{\phi-\psi}_U<\varepsilon$.
Then for each $s_j\in S\subset U$ we have $|\phi(s_j)-\psi(s_j)|\le \norm{\phi-\psi}_U<\varepsilon$.
\end{proof}

\subsection{Existence for Interpretation with Inclusion}
The density of rational functions immediately yields the existence of interpreters with inclusion. The construction is related to the problem of $p$-adic interpolation of continuous functions, but with the added requirement of dynamic compatibility on balls.

\begin{corollary}[Existence of Rational Interpreters with Inclusion]
\label{cor:existence_strong}
For any ball system $(\mathcal{B}, \tau, \mathcal{B}')$, there exists a rational function $\phi \in K(z)$ that interprets $\tau$ with inclusion.
\end{corollary}
\begin{proof}
Define $\psi \in \cO(U)$ by setting $\psi|_{B_i}(z)=b_{\tau(i)}$ for each $i$.
By Lemma~\ref{lem:disjoint_affinoids} this uniquely determines an analytic function on
$U=\bigsqcup_i B_i$.

Choose $\varepsilon < \min_i t_{\tau(i)}$. By Theorem~\ref{thm:runge}, there exists $\phi \in K(z)$ such that $|\phi(z) - \psi(z)| < \varepsilon$ for all $z \in U$. Then for $z \in B_i$:
\[
|\phi(z) - b_{\tau(i)}| \leq \max(|\phi(z) - \psi(z)|, |\psi(z) - b_{\tau(i)}|)=|\phi(z)-\psi(z)| \le \varepsilon < t_{\tau(i)}.
\]
Thus $\phi(B_i) \subseteq B_{\tau(i)}'$.
\end{proof}

\begin{corollary}[Pole-Free Rational Interpreters]
\label{cor:pole_free_with_control}
Let $K$ be a complete non-Archimedean field and let
\[
U = \bigsqcup_{i=1}^N B_i \subset \PP^1(K)
\]
be a finite disjoint union of closed balls. Let $\psi\in\cO(U)$ be the piecewise-constant analytic function defined by $\psi|_{B_i}\equiv b_{\tau(i)}$ (so $\psi$ is bounded on $U$). Then:

\begin{enumerate}
\item \textbf{(No poles on $U$).} If $\phi \in K(z)$ satisfies $\norm{\phi - \psi}_U < \infty$ (equivalently, $\norm{\phi}_U < \infty$ since $U$ is a compact affinoid), then $\phi$ has no poles on $U$, i.e., $\phi \in \cO(U)$.

\item \textbf{(Prescribing where poles may lie).} Let $V \subset \PP^1(K)$ be an admissible open such that $V \cap U = \varnothing$ and $V$ meets every connected component of $\PP^1(K) \setminus U$. Then for every $\varepsilon > 0$ there exists $\phi \in K(z)$, with all poles contained in $V$, such that $\norm{\phi - \psi}_U < \varepsilon$. In particular, the rational interpreter can be chosen analytic on $U$ and with poles contained in $V$.
\end{enumerate}
\end{corollary}
\begin{proof}
(1) If $\phi$ had a pole at $x \in U$, then $|\phi|$ would be unbounded on every neighborhood of $x$, hence $\norm{\phi - \psi}_U = \infty$ because $\psi$ is bounded on $U$. Thus $\phi$ has no poles on $U$, i.e., $\phi \in \cO(U)$.

(2) This is a standard Runge approximation statement on the rigid-analytic projective line: under the hypothesis that $V$ meets every connected component of the complement, the subalgebra of rational functions with poles in $V$ is dense in $\cO(U)$ for the sup-norm on $U$. Apply this to $\psi \in \cO(U)$ to obtain $\phi$ with poles in $V$ and $\norm{\phi - \psi}_U < \varepsilon$.
\end{proof}

\begin{corollary}[Pole-Free Interpretation for Ball Systems]
\label{cor:pole_free}
For any ball system $(\mathcal{B}, \tau, \mathcal{B}')$ with $U = \bigcup B_i$, there exists a rational function $\phi \in K(z)$ that interprets $\tau$ with inclusion and has no poles in $U$. Moreover, $\phi$ can be chosen analytic on $U$, i.e., with all poles in $\PP^1(K) \setminus U$.
\end{corollary}
\begin{proof}
This follows from Corollary~\ref{cor:pole_free_with_control} by taking $V = \PP^1(K) \setminus U$ (which is an admissible open) and applying the construction from Corollary~\ref{cor:existence_strong}.
\end{proof}

\begin{remark}[Runge Approximation with Prescribed Poles]
\label{rem:runge_with_poles_reference}
The statement in Corollary~\ref{cor:pole_free_with_control}(2) is a standard result in rigid analytic geometry and underlies approximation arguments in $p$-adic dynamics (see~\cite{BenedettoDynamics}). This is the rigid-analytic Runge theorem for a Runge pair $(U,V)$; see~\cite{BGR} (Runge pairs and admissible opens) or~\cite{FvdP} for the affinoid case.\footnote{The Archimedean analogue is Runge's theorem on approximation by rational functions with poles in a prescribed set.} The key hypothesis is that $V$ is an admissible open that meets every connected component of $\PP^1(K) \setminus U$; under this condition, the subalgebra of rational functions with poles contained in $V$ is dense in $\cO(U)$ for the sup-norm topology. 

Be careful: ``poles contained in a prescribed \emph{compact}'' is not true without additional hypotheses (e.g., Runge pair conditions); the natural object to prescribe is an admissible open $V$ meeting each complementary component. The domain $U$ here is a finite union of closed balls, which forms a Weierstrass domain (affinoid) in the sense of rigid analytic geometry.
\end{remark}

\begin{remark}
This result provides a qualitative answer to the \emph{inverse lifting problem}. While classical interpolation results (e.g., Mahler's theorem~\cite{Mahler1958}) address the representation of continuous functions, our construction ensures that the continuous dynamics mirrors the combinatorial transitions of the functional graph at a prescribed resolution $\varepsilon$. This connection between local analyticity and global combinatorial structure is a central theme in arithmetic dynamics~\cite{JonesManes2014, Poonen2014}.
\end{remark}

\begin{remark}[Alternative construction via $\varepsilon$-gluing]
\label{rem:alternative_methods}
A complementary construction is given in~\cite{PBNC2025}. Let $B = \bigcup_i B_{r_i}(a_i)$ be a finite union of disjoint rational balls and let $f_i$ be local rational or analytic maps with $f_i(B_{r_i}(a_i)) = B_{t_i}(b_i)$ for each $i$. Under the hypothesis that $f_i(B) \subset B_1(0)$ for all $i$, Theorem~4.2 of~\cite{PBNC2025} yields, for every $\varepsilon>0$, a global rational map $F_\varepsilon$ such that $F_\varepsilon(B_{r_i}(a_i))=B_{t_i}(b_i)$ and $|F_\varepsilon(z)-f_i(z)|<\varepsilon$ on $B_{r_i}(a_i)$ for each $i$. In the proof of Theorem~4.2, an explicit formula $F_\varepsilon(z)=\sum_i f_i(z)h_i(z)$ is given. See~\cite{PBNC2025} for the precise statement. The Runge approach in this section gives existence under minimal hypotheses (pairwise disjoint balls, any complete non-Archimedean field) but does not provide uniform degree or height bounds.
\end{remark}
 \section{Robust Exactness via Linear Dominance}
\label{sec:robust}

Interpretation with inclusion is achieved through approximation (Corollary~\ref{cor:existence_strong}); exact interpretation ($\phi(B_i) = B_{\tau(i)}'$) is a more delicate property. In this section, we show that exactness is robust whenever the map satisfies a local linear dominance condition.

\subsection{Linear Dominance and Ball Mapping}
\begin{definition}[Linear Dominance]
\label{def:dominance}
Let $f$ be an analytic function on $\Ball{a}{r}$ with expansion $f(z) = \sum_{k \ge 0} c_k (z-a)^k$. We say $f$ has \emph{linear dominance} on the ball if
\[
\max_{k \ge 2} |c_k| r^{k-1} < |c_1|.
\]
\end{definition}
\begin{remark}[Local similarity on the ball]
Under linear dominance, for $x,y\in \Ball{a}{r}$ we have
\[
|\phi(x)-\phi(y)| = |\phi'(a)|\,|x-y|.
\]
Thus the map is a similarity on the ball (Lipschitz constant $|\phi'(a)|$). This implication depends on linear dominance; in general, higher-order terms can change the local Lipschitz behavior.
\end{remark}
\begin{remark}[Intrinsic character]
This condition is intrinsic to the ball: it is invariant under affine isometries (conjugacy by $\sigma(z)=\alpha z+\beta$ with $|\alpha|=1$), so it does not depend on the choice of center; see Lemma~\ref{lem:conjugacy_invariance_affine_isometries} in Section~\ref{sec:closing_the_loop}.
\end{remark}

\begin{lemma}[Exactness Lemma]
\label{lem:dominance_image}
If $f$ has linear dominance on $\Ball{a}{r}$, then the image $f(\Ball{a}{r})$ is exactly the ball $\Ball{f(a)}{|f'(a)|r}$.
This lemma provides the proof-level mechanism behind Remark~\ref{rem:image_ball}.
\end{lemma}
\begin{proof}
For any $z \in \Ball{a}{r}$, let $\delta = z-a$ with $|\delta| \leq r$. Then $|f(z) - f(a)| = |\sum_{k \ge 1} c_k \delta^k|$. By Definition~\ref{def:dominance}, for $k \ge 2$:
\[
|c_k \delta^k| = |c_k| |\delta|^{k-1} |\delta| \leq (|c_k| r^{k-1}) |\delta| < |c_1| |\delta|.
\]
Thus the $k=1$ term strictly dominates the sum, and the ultrametric equality gives $|f(z) - f(a)| = |c_1 \delta| = |c_1| |z-a|$. This shows the inclusion into $\Ball{f(a)}{|c_1|r}$. To prove surjectivity, fix $w \in \Ball{f(a)}{|c_1|r}$ and consider the map $T(z) = z - c_1^{-1}(f(z) - w)$. For any $z, z' \in \Ball{a}{r}$, we have
\[
T(z) - T(z') = (z - z') - c_1^{-1}(f(z) - f(z')) = -c_1^{-1} \sum_{k \ge 2} c_k ((z-a)^k - (z'-a)^k).
\]
For $z, z' \in \Ball{a}{r}$ we have $|z-a|, |z'-a| \leq r$. Using the identity $X^k - Y^k = (X-Y) \sum X^j Y^{k-1-j}$, we get $|(z-a)^k - (z'-a)^k| \leq |z-z'| r^{k-1}$. Thus,
\[
|T(z) - T(z')| \leq |c_1|^{-1} |z-z'| \max_{k \ge 2} |c_k| r^{k-1} = \theta |z-z'|,
\]
where $\theta = |c_1|^{-1} \max_{k \ge 2} |c_k| r^{k-1} < 1$ by linear dominance. Hence $T$ is a strict contraction. Since $|T(a) - a| = |c_1|^{-1} |f(a) - w| \leq r$ and $T$ is $\theta$-Lipschitz on $\Ball{a}{r}$ with $\theta<1$, for any $z\in \Ball{a}{r}$ we have
\[
|T(z)-a|\le \max\big(|T(z)-T(a)|,\;|T(a)-a|\big)\le \max(\theta|z-a|,r)\le r.
\]
Hence $T$ maps $\Ball{a}{r}$ to itself. The closed ball $\Ball{a}{r}$ is complete (because $K$ is complete and a closed ball in an ultrametric space is complete). By the non-Archimedean contraction mapping principle (see~\cite{BenedettoDynamics}; the rigid-analytic formulation is in \cite{BGR}),\footnote{In the Archimedean setting the analogous result is Banach's contraction principle.} $T$ has a unique fixed point $z^* \in \Ball{a}{r}$, which satisfies $f(z^*) = w$.
\end{proof}

\subsection{Robustness Theorem}
We can now prove that exactness is an open property in the sup-norm.

\begin{lemma}[Ultrametric Ball Coincidence]
\label{lem:ball_coincidence}
In an ultrametric field, for balls of the same radius $r$ we have $B(c,r)=B(c',r)$ if and only if $|c-c'|\le r$.
\end{lemma}
\begin{proof}
($\Rightarrow$) If $B(c,r)=B(c',r)$, then $c'\in B(c,r)$, so $|c-c'|\le r$.

($\Leftarrow$) If $|c-c'|\le r$, then $c'\in B(c,r)$. For any $z\in B(c,r)$ we have $|z-c|\le r$; by the ultrametric inequality $|z-c'|\le \max(|z-c|,|c-c'|)\le r$, so $z\in B(c',r)$. Thus $B(c,r)\subseteq B(c',r)$. Symmetrically $B(c',r)\subseteq B(c,r)$, hence $B(c,r)=B(c',r)$.
\end{proof}

\begin{theorem}[Robust Exactness]
\label{thm:robust_exact}
Let $(\mathcal{B}, \tau, \mathcal{B}')$ be a ball system. Let $\psi$ be the piecewise affine interpreter from Theorem~\ref{thm:analytic_existence}, where $|u_i| = t_{\tau(i)}/r_i$. Let $\varepsilon$ satisfy
\[
\varepsilon < \min_i |u_i| r_i = \min_i t_{\tau(i)}.
\]
Then any rational function $\phi$ such that $\norm{\phi - \psi}_U < \varepsilon$ is an exact interpreter of $\tau$.
\end{theorem}
\begin{proof}
Fix $i$ and consider $\phi$ on $B_i$. By Lemma~\ref{lem:cauchy_diff}:
\begin{enumerate}
    \item $|\phi'(a_i) - u_i| \leq \varepsilon/r_i < |u_i|$. By the ultrametric triangle inequality, $|\phi'(a_i)| = |u_i|$.
    \item For $k \ge 2$, $|c_k(\phi)| \leq \varepsilon/r_i^k$, so $|c_k(\phi)| r_i^{k-1} \leq \varepsilon/r_i < |u_i| = |\phi'(a_i)|$.
\end{enumerate}
Thus $\phi$ satisfies linear dominance on $B_i$. By Lemma~\ref{lem:dominance_image},
\[
\phi(B_i) = \Ball{\phi(a_i)}{|\phi'(a_i)| r_i} = \Ball{\phi(a_i)}{t_{\tau(i)}}.
\]
Since $\psi(a_i)=b_{\tau(i)}$ and $\norm{\phi-\psi}_U<\varepsilon<t_{\tau(i)}$, we obtain
\[
|\phi(a_i)-b_{\tau(i)}|
=
|\phi(a_i)-\psi(a_i)|
\le \norm{\phi-\psi}_U
< \varepsilon
< t_{\tau(i)}.
\]
In particular, $|\phi(a_i)-b_{\tau(i)}|\le t_{\tau(i)}$. By Lemma~\ref{lem:ball_coincidence}, we have
\[
\Ball{\phi(a_i)}{t_{\tau(i)}}=\Ball{b_{\tau(i)}}{t_{\tau(i)}}=B'_{\tau(i)}.
\]
Hence $\phi(B_i) = B_{\tau(i)}'$.
\end{proof}

\begin{corollary}[Stability of Multipliers]
\label{cor:multiplier_stability}
Let $\phi$ be an exact interpreter as in Theorem~\ref{thm:robust_exact}. If $x$ is a fixed point of $F$ (so $\tau(i)=i$) and $\alpha \in B_i$ is a fixed point of $\phi$, then the $p$-adic multiplier $\lambda = \phi'(\alpha)$ satisfies:
\[
|\lambda| = |u_i| = \frac{t_i}{r_i}.
\]
In particular, the stability type (contractive, expansive, or indifferent) of the fixed point $\alpha$ is determined solely by the affine model $\psi$ and is invariant under $\varepsilon$-perturbations.
\end{corollary}
\begin{proof}
By the linear dominance proven in Theorem~\ref{thm:robust_exact}, we have $|\phi'(z)| = |u_i|$ for all $z \in B_i$. Since $\alpha \in B_i$, $|\phi'(\alpha)| = |u_i|$.
\end{proof}

\begin{corollary}[Existence of Exact Interpreters via Runge with Interpolation]
\label{cor:exact_via_runge_interpolation}
For any ball system $(\mathcal{B}, \tau, \mathcal{B}')$ with $U = \bigcup B_i$, there exists a rational function $\phi \in K(z)$ that exactly interprets $\tau$ (i.e., $\phi(B_i) = B_{\tau(i)}'$ for all $i$).
\end{corollary}
\begin{proof}
Let $\psi$ be the piecewise affine interpreter from Theorem~\ref{thm:analytic_existence}, and let $S = \{a_1, \dots, a_N\}$ be the set of centers. Choose $\varepsilon < \min_i t_{\tau(i)}$. By Lemma~\ref{lem:runge_finite_constraints}, there exists $\phi \in K(z)$ such that $\norm{\phi - \psi}_U < \varepsilon$, hence $|\phi(a_i) - \psi(a_i)| < \varepsilon$ for all $i$. Since $\psi(a_i) = b_{\tau(i)}$ by construction, we have $|\phi(a_i) - b_{\tau(i)}| < \varepsilon \leq t_{\tau(i)}$ for all $i$. By Theorem~\ref{thm:robust_exact}, $\phi$ is an exact interpreter of $\tau$.
\end{proof}

\begin{remark}
The relaxed condition $|\phi(a_i) - b_{\tau(i)}| \leq t_{\tau(i)}$ in Theorem~\ref{thm:robust_exact} is more natural than requiring exact equality. Corollary~\ref{cor:exact_via_runge_interpolation} shows that this condition can be achieved via Lemma~\ref{lem:runge_finite_constraints}, which gives control at finitely many points ($|\phi(s_j)-\psi(s_j)| \le \varepsilon$) without requiring exact interpolation. The ultrametric property ensures that balls of the same radius coincide if their centers are within that radius.
\end{remark}
 \section{Arithmetic Extensions: Witt Vectors and Profinite Coherence}
\label{sec:composite_alphabets}

\subsection{Motivation: Beyond Prime Alphabets}
Up to now, we have assumed an alphabet of size $p$, leading to dynamics on $\Zp$. However, in most discrete systems, the alphabet size $q$ is not necessarily prime. We now generalize by showing that the hierarchy of cylinder balls is equivalent to the hierarchy of \emph{congruence classes} in the ring of Witt vectors.

\subsection{From Multi-state Alphabets to Witt Cylinders}
\label{subsec:witt-cylinders}

Fix a prime $p$ and $q=p^k$. Let $K/\Qp$ be the unramified extension of degree $k$ with valuation ring $\cO_K$ and residue field $\kappa \cong \F_q$. A fundamental structural result in local field theory (see Serre~\cite{SerreLocalFields} and Hazewinkel~\cite{HazewinkelWitt}) is the identification $\cO_K \cong W(\kappa)$ (the Witt identification in the unramified case):
\begin{equation}
\label{eq:OK-as-Witt}
\cO_K \ \cong\ W(\kappa),
\end{equation}
where $W(\kappa)$ is the ring of $p$-typical Witt vectors over $\kappa$. Truncation modulo $p^N$ then corresponds to the ring of truncated Witt vectors of length $N$:
\begin{equation}
\label{eq:trunc}
\cO_K/p^N\cO_K \ \cong\ W_N(\kappa).
\end{equation}

\begin{remark}[Field Base Convention]
\label{rem:field_base}
In this section, we work with two related fields:
\begin{itemize}
    \item $K_0$ denotes the maximal unramified complete extension of $\mathbb{Q}_p$ with residue field
    $\kappa=\overline{\mathbb{F}}_p$ and ring of integers $\mathcal{O}_{K_0}\cong W(\kappa)$.
    In the parts of the paper where we use the Witt-vector description via the fixed identification $\cO_{K_0}\cong W(\kappa)$ (unramified base field $K_0/\mathbb{Q}_p$), we typically take $K=K_0$.
    \item In Section~\ref{sec:existence}, we require $K$ to be algebraically closed with divisible value group,
    so we work in $K=\mathbb{C}_p$, and we fix once and for all an embedding $K_0\hookrightarrow\mathbb{C}_p$.
    This yields the embedding $\mathcal{O}_{K_0}\hookrightarrow \mathcal{O}_{\mathbb{C}_p}$ (not an isomorphism).
    \item The truncation map $\pi_N:\mathcal{O}_{K_0}\twoheadrightarrow \mathcal{O}_{K_0}/p^N\mathcal{O}_{K_0}\cong W_N(\kappa)$
    is the reduction in the unramified setting. When working in $K=\mathbb{C}_p$ (Section~\ref{sec:existence}),
    we view $\mathcal{O}_{K_0}$ as a distinguished subring of $\mathcal{O}_{\mathbb{C}_p}$ via the fixed embedding,
    and we apply $\pi_N$ only after restricting to elements of $\mathcal{O}_{K_0}$ (in particular, to points lying in Witt cylinders).
\end{itemize}
\textbf{Global convention:} From now on, we fix an embedding $K_0 \hookrightarrow \mathbb{C}_p$ and work with the fixed identification
$\mathcal{O}_{K_0} \cong W(\kappa)$ for the unramified extension $K_0/\mathbb{Q}_p$ (with $\kappa=\overline{\mathbb{F}}_p$).
The \emph{Witt cylinders} $\mathbb{B}_N(x)=\pi_N^{-1}(x)$ for $x\in W_N(\kappa)$ are congruence balls of depth $N$ in $\mathcal{O}_{K_0}$,
viewed naturally as subsets of $\mathbb{C}_p$ via the fixed embedding. In particular, while analytic constructions in
Section~\ref{sec:existence} take place over $\mathbb{C}_p$, all cylinder constraints are imposed on subsets of
$\mathcal{O}_{K_0}\subset \mathbb{C}_p$; no reduction map on $\mathcal{O}_{\mathbb{C}_p}$ is used.
The relationship between $K_0$ and $K=\mathbb{C}_p$ is summarized by the commutative diagram
\[
\begin{tikzcd}
K_0 \arrow[r, hook] & \mathbb{C}_p \\
\mathcal{O}_{K_0} \arrow[r, hook] \arrow[d, "\pi_N"] & \mathcal{O}_{\mathbb{C}_p} \\
W_N(\kappa) &
\end{tikzcd}
\]
where the horizontal arrows are induced by the fixed field embedding $K_0\hookrightarrow \mathbb{C}_p$, and the only truncation map
$\pi_N$ used in the paper is the reduction on $\mathcal{O}_{K_0}$ (equivalently, on $W(\kappa)$).
\end{remark}

For each configuration $x \in W_N(\kappa)$, we define the associated \emph{Witt cylinder ball} $\mathbb{B}_N(x) := \pi_N^{-1}(x)$, which generalizes the state cylinders defined in Section~\ref{sec:encoding} to non-prime alphabets. In the special case $k=1$ (where $q=p$), we have $\cO_K = \mathbb{Z}_p$ and these Witt cylinders coincide with the state cylinders $B_x$ defined in Equation~\eqref{eq:encoding}. These cylinders form a partition of $\cO_K$ at resolution depth $N$.

\begin{remark}[Unramified vs.\ ramified]
Ball-semantics and interpretation (Section~\ref{sec:semantics}) make sense over any complete non-Archimedean field; the Witt-cylinder coding and Teichmüller representatives used here are set up for \emph{unramified} extensions. For ramified extensions, the residue ring is not of the form $\mathbb{Z}/p^k\mathbb{Z}$ in general; the ramified case is deferred (Remark~\ref{rem:dcrt_ramified_scope}).
\end{remark}

\subsection{Compatible Towers and Profinite Dynamics}
\label{subsec:towers-profinite}

A discrete rule $f_N$ is often just a truncation of a mechanism operating at higher resolutions. This is formalized by compatible towers.

\begin{definition}[Compatible Tower]
A family of maps $\{f_N: W_N(\kappa) \to W_N(\kappa)\}_{N \ge 1}$ is \emph{compatible} if for all $M \le N$, the truncation map $\tau_{N,M}: W_N(\kappa) \to W_M(\kappa)$ satisfies $\tau_{N,M} \circ f_N = f_M \circ \tau_{N,M}$.
\end{definition}

\begin{theorem}[Profinite Inverse Limit]
\label{thm:inverse-limit}
By the universal property of inverse limits~\cite{RibesZalesskii}, any compatible tower $\{f_N\}$ induces a unique continuous map $\widehat{f}: \cO_K \to \cO_K$ satisfying the coherence condition:
\begin{equation}
\label{eq:profinite-coherence}
\pi_N \circ \widehat{f} = f_N \circ \pi_N \quad \text{for all } N \ge 1.
\end{equation}
\end{theorem}

This allows us to view the inverse limit as the domain for the dynamical system: the discrete $f_N$ are finite-resolution shadows of a global continuous dynamics $\widehat{f}$ on the unramified base $\mathcal{O}_{K_0} \cong W(\kappa)$ (see Remark~\ref{rem:field_base}). A rational map $\phi$ that interprets each $f_N$ with inclusion (Lemma~\ref{lem:strong_commutation}) induces this same limit map $\widehat{f}$; in that sense $\phi$ realizes the global dynamics.

\subsection{Vertical refinement: Henselian lifting of exact cycles}
\label{subsec:hensel_cycles}

The Dynamic Chinese Remainder Theorem provides a \emph{horizontal} decomposition of discrete dynamics across coprime moduli. A complementary phenomenon is \emph{vertical} refinement along the inverse system $(\mathbb{Z}/p^n\mathbb{Z})_{n\ge 1}$, which corresponds to the profinite limit $\Zp$. In this direction, one can lift \emph{exact} periodic cycles from the residue field to $\Zp$ under the non-degeneracy condition $(\widetilde{\phi}^{\circ m})'(\bar{x})\neq 1$ (Theorem~\ref{thm:hensel_exact_cycle}). A sufficient condition for reduction to commute with composition is that $\phi$ be given by a strictly convergent power series in the Tate algebra $\cO_K\langle z\rangle$ (or more generally by an analytic function with integral coefficients so that coefficientwise reduction defines $\widetilde\phi$); the following lemma makes this checkable.

\begin{lemma}[Composition-compatible reduction]
\label{lem:reduction_composition}
Let $\cO_K$ be the valuation ring of a complete discretely valued field $K$ with residue field $\kappa$. Coefficientwise reduction is a ring homomorphism from the Tate algebra $\cO_K\langle z\rangle$ (strictly convergent power series with coefficients in $\cO_K$) to $\kappa[z]$; hence $\widetilde{(\phi\circ\psi)}=\widetilde\phi\circ\widetilde\psi$ for $\phi,\psi\in\cO_K\langle z\rangle$ (and thus $\widetilde{\phi^{\circ m}}=\widetilde\phi^{\circ m}$ for all $m\ge 1$). See, e.g.,~\cite{BGR} or~\cite{BenedettoDynamics} for the rigid setting.
\end{lemma}

\begin{theorem}[Henselian lifting of exact periodic points]
\label{thm:hensel_exact_cycle}
Let $K$ be a complete discretely valued non-Archimedean field with valuation ring $\cO_K$, maximal ideal $\mathfrak{m}$, and residue field $\kappa=\cO_K/\mathfrak{m}$. Let $\phi\colon \cO_K \to \cO_K$ be an analytic map induced by a power series in $\cO_K\langle z\rangle$ (so reduction $\widetilde{\phi}\colon \kappa \to \kappa$ is defined and Lemma~\ref{lem:reduction_composition} applies). Fix $m\ge 1$ and assume $\bar{x}\in\kappa$ has \emph{exact} period $m$ under $\widetilde{\phi}$:
\[
\widetilde{\phi}^{\circ m}(\bar{x})=\bar{x},
\qquad
\widetilde{\phi}^{\circ d}(\bar{x})\neq \bar{x}\ \text{for all}\ d\mid m,\ d<m.
\]
Assume moreover the non-degeneracy condition
\[
(\widetilde{\phi}^{\circ m})'(\bar{x})\neq 1 \quad \text{in }\kappa.
\]
(Equivalently, $\bar{x}$ is a \emph{simple root} of $\widetilde{F}(z)=\widetilde{\phi}^{\circ m}(z)-z$, since $\widetilde{F}'(\bar{x})=(\widetilde{\phi}^{\circ m})'(\bar{x})-1$.)
Then there exists a \emph{unique} $x\in\cO_K$ reducing to $\bar{x}$ such that $\phi^{\circ m}(x)=x$. Furthermore, $x$ has \emph{exact} period $m$ under $\phi$.
\end{theorem}
\begin{proof}
Define $F(z)=\phi^{\circ m}(z)-z$ on $\cO_K$. Since $\phi(\cO_K)\subseteq \cO_K$, we have $F(\cO_K)\subseteq \cO_K$ and thus $F$ reduces to a map $\widetilde{F}\colon \kappa\to\kappa$. By Lemma~\ref{lem:reduction_composition},
\[
\widetilde{F}(\bar{z})=\widetilde{\phi^{\circ m}}(\bar{z})-\bar{z}
=\widetilde{\phi}^{\circ m}(\bar{z})-\bar{z}.
\]
Hence $\widetilde{F}(\bar{x})=0$.

Differentiating gives $F'(z)=(\phi^{\circ m})'(z)-1$, so after reduction,
\[
\widetilde{F}'(\bar{x})=(\widetilde{\phi}^{\circ m})'(\bar{x})-1\neq 0
\]
by the non-degeneracy assumption. Therefore $\bar{x}$ is a \emph{simple root} of $\widetilde{F}$. By the classical Hensel lemma for simple roots (see, e.g.,~\cite{BGR} or~\cite{BenedettoDynamics}), there exists a unique $x\in\cO_K$ with $x\bmod\mathfrak{m}=\bar{x}$ and $F(x)=0$, i.e.\ $\phi^{\circ m}(x)=x$.

To show the period is \emph{exactly} $m$, suppose $\phi^{\circ d}(x)=x$ for some proper divisor $d\mid m$, $d<m$. Reducing modulo $\mathfrak{m}$ yields $\widetilde{\phi}^{\circ d}(\bar{x})=\bar{x}$, contradicting the exact-period assumption. Hence no proper divisor $d<m$ satisfies $\phi^{\circ d}(x)=x$, and the period of $x$ is $m$.
\end{proof}

\begin{remark}[What the equation $\phi^{\circ m}(x)=x$ does and does not guarantee]
\label{rem:period_divides}
The equation $\phi^{\circ m}(x)=x$ only implies that the minimal period of $x$ \emph{divides} $m$. Exactness requires additional input. In Theorem~\ref{thm:hensel_exact_cycle} exactness is inherited from the residue dynamics: if a shorter period existed upstairs, it would persist after reduction, contradicting the exact period of $\bar{x}$ in $\kappa$.
\end{remark}

\begin{remark}[Degenerate case]
\label{rem:hensel_degenerate}
If $(\widetilde{\phi}^{\circ m})'(\bar{x})=1$ in $\kappa$, then $\bar{x}$ is not a simple root of $\widetilde{F}$ and the simple-root Hensel lemma does not apply. In this degenerate situation, lifts may fail to exist or may be non-unique, and periodic structure can split at higher levels. A systematic treatment requires higher-order analysis beyond the scope of this paper.
\end{remark}

\begin{remark}[Why this does not lift all cycles simultaneously]
\label{rem:hensel_local_only}
The theorem is \emph{local} (per periodic point or per cycle) and hinges on the simple-root condition for $F(z)=\phi^{\circ m}(z)-z$ at the chosen residue class. Simultaneously lifting all periodic points of a given level typically requires additional global separation hypotheses (e.g.\ separation of balls and no collisions among lifts). For this reason, inverse-limit realization of the full tower is left as a direction (Route~2 in Appendix~\ref{app:profinite_addendum}), not as an automatic consequence of the theorem.
\end{remark}

\begin{corollary}[Tower of cycles and limit in $\mathbb{Z}_p$]
\label{cor:hensel_tower}
Let $K=\Qp$ and $\phi\colon \Zp\to\Zp$ be analytic. Suppose $\bar{x}\in\Fp$ satisfies the hypotheses of Theorem~\ref{thm:hensel_exact_cycle} for some $m\ge 1$. Then there is a unique sequence $(x_n)_{n\ge 1}$ with $x_n\in\mathbb{Z}/p^n\mathbb{Z}$ such that $x_{n+1}\equiv x_n\pmod{p^n}$, $\phi^{\circ m}(x_n)\equiv x_n\pmod{p^n}$, and $x_1=\bar{x}$. The limit $x=\lim_n x_n\in\Zp$ exists and is the unique periodic point of exact period $m$ lifting $\bar{x}$.
\end{corollary}
\begin{proof}
Apply Theorem~\ref{thm:hensel_exact_cycle} at each level $n$: the non-degeneracy condition ensures $F'(x)\in\mathbb{Z}_p^\times$ at the unique lift, so iterated Hensel (or Newton--Hensel) yields the coherent sequence $(x_n)$ and its limit in $\Zp$.
\end{proof}

In the interpreter setting of Sections~\ref{sec:encoding}--\ref{sec:robust}, the state cylinders (balls of radius $p^{-n}$) partition $\Zp$ at each resolution $n$; as $n\to\infty$ these balls shrink to points. Theorem~\ref{thm:hensel_exact_cycle} and Corollary~\ref{cor:hensel_tower} formalize when \emph{exact} periodic structure is preserved along this vertical refinement: under non-degeneracy of the multiplier, a residual $m$-cycle lifts to a unique $m$-cycle in $\Zp$. The promises are modest (one cycle at a time; degenerate case excluded) and the logic is self-contained.

\subsection{Semantics as Congruence Control}
\label{subsec:semantics-congruence}

The interpretation semantics can be rephrased as conditions on the induced maps on quotients. Let $f_N: W_N(\kappa) \to W_N(\kappa)$ be a discrete rule, where $\pi_N: \cO_K \to W_N(\kappa)$ is the truncation map modulo $p^N$.

\begin{lemma}[Equivalence of Interpretation with Inclusion and Commutation]
\label{lem:strong_commutation}
Let $\phi: \cO_K \to \cO_K$ be a continuous map and $f_N: W_N(\kappa) \to W_N(\kappa)$ a discrete rule. Then $\phi$ interprets $f_N$ with inclusion if and only if $\pi_N \circ \phi = f_N \circ \pi_N$ on $\cO_K$.
\end{lemma}
\begin{proof}
Recall that $\phi$ interprets $f_N$ with inclusion if for every $x \in W_N(\kappa)$, we have $\phi(\mathbb{B}_N(x)) \subseteq \mathbb{B}_N(f_N(x))$, where $\mathbb{B}_N(x) = \pi_N^{-1}(x)$.

($\Rightarrow$) Suppose $\phi$ interprets $f_N$ with inclusion. For any $z \in \cO_K$, let $x = \pi_N(z)$, so $z \in \mathbb{B}_N(x)$. By interpretation with inclusion, $\phi(z) \in \mathbb{B}_N(f_N(x)) = \pi_N^{-1}(f_N(x))$, which means $\pi_N(\phi(z)) = f_N(x) = f_N(\pi_N(z))$. Since $z$ was arbitrary, $\pi_N \circ \phi = f_N \circ \pi_N$.

($\Leftarrow$) Suppose $\pi_N \circ \phi = f_N \circ \pi_N$. For any $x \in W_N(\kappa)$ and $z \in \mathbb{B}_N(x)$, we have $\pi_N(z) = x$, so $\pi_N(\phi(z)) = f_N(\pi_N(z)) = f_N(x)$. This means $\phi(z) \in \pi_N^{-1}(f_N(x)) = \mathbb{B}_N(f_N(x))$. Since $z$ was arbitrary in $\mathbb{B}_N(x)$, we conclude $\phi(\mathbb{B}_N(x)) \subseteq \mathbb{B}_N(f_N(x))$.
\end{proof}

\begin{remark}[Semantics as Congruences]
\label{rem:semantics-witt}
By Lemma~\ref{lem:strong_commutation}, the interpretation notions of Section~\ref{sec:semantics} can be rephrased as follows. $\phi$ \emph{interprets} $f_N$ if for every $x$, $\phi(\mathbb{B}_N(x)) \cap \mathbb{B}_N(f_N(x)) \neq \varnothing$; by the ultrametric nesting property, for each $x$ the interpretation type at that cylinder is contractive, indifferent, or expansive according as $\phi(\mathbb{B}_N(x)) \subsetneq \mathbb{B}_N(f_N(x))$, $\phi(\mathbb{B}_N(x)) = \mathbb{B}_N(f_N(x))$, or $\mathbb{B}_N(f_N(x)) \subsetneq \phi(\mathbb{B}_N(x))$.
\begin{enumerate}
    \item \textbf{Interpretation with inclusion:} $\phi$ interprets $f_N$ with inclusion (contractive or indifferent at every cylinder) if and only if $\pi_N \circ \phi = f_N \circ \pi_N$ on $\mathcal{O}_{K_0}$.
    \item \textbf{Exact interpretation:} $\phi$ exactly interprets $f_N$ (indifferent at every cylinder) if and only if $\phi$ interprets $f_N$ with inclusion and is surjective on each cylinder. Surjectivity is guaranteed by the linear dominance criterion (see Section~\ref{sec:robust}):
    \begin{equation}
    \label{eq:dominance-check}
    \max_{k \ge 2} |c_k| r^{k-1} < |c_1|.
    \end{equation}
\end{enumerate}
\end{remark}

\subsection{Global Decomposition via Dynamic CRT}
\label{subsec:dcrt}

Let $m = \prod_{i=1}^s p_i^{k_i}$ be the prime factorization of the alphabet size, and set
$R_m = \mathbb{Z}/m\mathbb{Z}$ and $R_i = \mathbb{Z}/p_i^{k_i}\mathbb{Z}$. The classical Chinese Remainder Theorem gives a ring isomorphism
$\Theta: R_m \xrightarrow{\sim} \prod_{i=1}^s R_i$. By \emph{Dynamic} CRT we mean the following: for a congruence-preserving system
$F:R_m\to R_m$, the induced discrete dynamics (equivalently, the functional graph of $F$) factors through $\Theta$ into componentwise dynamics
$F_i:R_i\to R_i$, and this factorization is the one lifted in the product phase space constructed below. 

An arbitrary map $f:R_m\to R_m$ need not respect the CRT decomposition and may entangle the prime-power components (e.g., arbitrary look-up tables that mix residues modulo different primes need not be congruence-preserving).
We therefore restrict attention to maps that are compatible with reduction modulo every divisor of $m$; in particular, this class contains all maps induced by polynomials (Lemma~\ref{lem:poly_cp}).
Such maps are called \emph{congruence-preserving}; the DCRT applies to them, but not to every discrete dynamical system on $\mathbb{Z}/m\mathbb{Z}$.

\begin{definition}[Congruence-Preserving Map]
\label{def:congruence_preserving}
A map $f: R_m \to R_m$ is \emph{congruence-preserving} if for every divisor $d \mid m$ and all $x, y \in R_m$,
\[
x \equiv y \pmod{d} \quad \Rightarrow \quad f(x) \equiv f(y) \pmod{d}.
\]
We denote by $\End_{\mathrm{cp}}(R_m)$ the set of all congruence-preserving endomorphisms of $R_m$.
\end{definition}
\noindent
Equivalently, congruence-preservation is the statement that $f$ factors through every reduction modulo a divisor of $m$.
For each divisor $d\mid m$, let
\[
\pi_d: R_m \longrightarrow R_d=\mathbb{Z}/d\mathbb{Z}
\]
denote the reduction map $\pi_d$. Then $f$ is congruence-preserving if and only if for every $d\mid m$ there exists a (necessarily unique) map
\[
f_d: R_d \longrightarrow R_d
\]
such that
\[
\pi_d\circ f \;=\; f_d\circ \pi_d.
\]
In particular, when $m=\prod_{i=1}^s p_i^{k_i}$ and $\Theta: R_m \xrightarrow{\sim} \prod_{i=1}^s R_i$ is the CRT isomorphism with $R_i=\mathbb{Z}/p_i^{k_i}\mathbb{Z}$, write $\mathrm{pr}_i$ for the projection to the $i$-th factor and set
\[
\rho_i := \mathrm{pr}_i\circ \Theta: R_m \longrightarrow R_i.
\]
If $f$ is congruence-preserving, then for each $i$ the rule
\[
f_i(\rho_i(x)) := \rho_i(f(x))
\]
is well-defined and yields a map $f_i:R_i\to R_i$; moreover,
\[
f \;=\; \Theta^{-1}\circ (f_1\times\cdots\times f_s)\circ \Theta.
\]

\begin{lemma}[Polynomial Maps are Congruence-Preserving]
\label{lem:poly_cp}
Every polynomial $P(X) \in \mathbb{Z}[X]$ induces a congruence-preserving map $f: R_m \to R_m$ via $f(x) = P(x) \bmod m$.
\end{lemma}
\begin{proof}
Fix a divisor $d\mid m$. If $x\equiv y\pmod d$, then $x-y\in dR_m$, hence
\[
x^k-y^k=(x-y)(x^{k-1}+x^{k-2}y+\cdots+y^{k-1})\in dR_m
\]
for every $k\ge 1$, so $x^k\equiv y^k\pmod d$. Writing $P(X)=\sum_{k=0}^n a_k X^k$ with $a_k\in\mathbb{Z}$, we obtain
\[
P(x)-P(y)=\sum_{k=0}^n a_k(x^k-y^k)\in dR_m,
\]
hence $P(x)\equiv P(y)\pmod d$. Therefore the induced map $f(x)=P(x)\bmod m$ is congruence-preserving.
\end{proof}

\begin{definition}[Product of Functional Graphs]
\label{def:graph_product}
For functional graphs $G_{f_1} = (X_1, E_1)$ and $G_{f_2} = (X_2, E_2)$, their \emph{direct product} $G_{f_1} \times G_{f_2}$ is the functional graph with vertex set $X_1 \times X_2$ and edges $((x_1, x_2), (f_1(x_1), f_2(x_2)))$ for all $(x_1, x_2) \in X_1 \times X_2$.
\end{definition}

\begin{lemma}[CRT Assembly Preserves Congruence-Preserving Property]
\label{lem:crt_preserves_cp}
Let $m = \prod_{i=1}^s p_i^{k_i}$ with $(p_i)$ distinct primes, and write $R_m = \mathbb{Z}/m\mathbb{Z}$ and $R_i = \mathbb{Z}/p_i^{k_i}\mathbb{Z}$. Let $\Theta: R_m \xrightarrow{\sim} \prod_{i=1}^s R_i$ be the Chinese remainder isomorphism. Suppose $f_i \in \End_{\mathrm{cp}}(R_i)$ for each $i$, where ``congruence-preserving'' means: for every divisor $d_i \mid p_i^{k_i}$,
\[
x_i \equiv y_i \pmod{d_i} \quad \Longrightarrow \quad f_i(x_i) \equiv f_i(y_i) \pmod{d_i}.
\]
Define $f: R_m \to R_m$ by
\[
f := \Theta^{-1} \circ \left(\prod_{i=1}^s f_i\right) \circ \Theta.
\]
Then $f \in \End_{\mathrm{cp}}(R_m)$, i.e., for every divisor $d \mid m$,
\[
x \equiv y \pmod{d} \quad \Longrightarrow \quad f(x) \equiv f(y) \pmod{d}.
\]
\end{lemma}
\begin{proof}
Let $d \mid m$ and write $d = \prod_{i=1}^s p_i^{e_i}$ with $0 \leq e_i \leq k_i$. Under $\Theta$ we have the equivalence
\[
x \equiv y \pmod{d} \quad \Longleftrightarrow \quad x_i \equiv y_i \pmod{p_i^{e_i}} \text{ for all } i.
\]
By hypothesis, each $f_i$ preserves congruences modulo $p_i^{e_i}$, hence $f_i(x_i) \equiv f_i(y_i) \pmod{p_i^{e_i}}$ for all $i$. Applying $\Theta^{-1}$ yields $f(x) \equiv f(y) \pmod{d}$.
\end{proof}

\begin{remark}[Notation Convention]
\label{rem:notation_depth_vs_modulus}
To avoid collision with Witt truncation depth, we use $\ell$ (or $N$) for $W_\ell(\kappa)$ and $m$ for the CRT modulus $R_m = \mathbb{Z}/m\mathbb{Z}$.
\end{remark}

\begin{theorem}[Dynamic CRT]
\label{thm:dcrt}
Let $\End_{\mathrm{cp}}(R_m)$ be the set of congruence-preserving endomorphisms of $R_m$ (see Definition~\ref{def:congruence_preserving}). There is a bijection $\End_{\mathrm{cp}}(R_m) \cong \prod_{i=1}^s \End_{\mathrm{cp}}(R_i)$ induced by $\Theta$. For any such $f$, its functional graph $G_f$ is isomorphic to the direct product of the functional graphs of its components:
\[
G_f \cong G_{f_1} \times G_{f_2} \times \dots \times G_{f_s}.
\]
\end{theorem}
\begin{proof}
Let $f \in \End_{\mathrm{cp}}(R_m)$. For each $i$, the map $f$ induces a well-defined map $f_i: R_i \to R_i$ given by $f_i(x \pmod{p_i^{k_i}}) = f(x) \pmod{p_i^{k_i}}$. The congruence-preserving property (with $d = p_i^{k_i}$) ensures that $f_i$ does not depend on the choice of representative $x \in R_m$, and that $f_i$ is itself congruence-preserving on $R_i$. The Chinese Remainder Theorem ensures that $f$ is uniquely determined by the tuple $(f_1, \dots, f_s)$. 

Conversely, by Lemma~\ref{lem:crt_preserves_cp}, any tuple $(f_1, \dots, f_s)$ of congruence-preserving maps on the components defines a congruence-preserving map on $R_m$ via $\Theta^{-1}$.

For the graph isomorphism, since $f(x)$ corresponds to $(f_1(x_1), \dots, f_s(x_s))$ in product coordinates under $\Theta$, the transition $(x, f(x))$ in $G_f$ corresponds to the product of transitions $(x_i, f_i(x_i))$ in each $G_{f_i}$. This establishes the isomorphism $G_f \cong G_{f_1} \times \dots \times G_{f_s}$ as defined in Definition~\ref{def:graph_product}.
\end{proof}

\begin{example}[Non-Congruence-Preserving Map]
\label{ex:non_cp}
Not every map $f: R_m \to R_m$ is congruence-preserving. For instance, consider $m = 6 = 2 \cdot 3$ and identify
$R_6=\mathbb{Z}/6\mathbb{Z}$ with the set of standard representatives $\{0,1,2,3,4,5\}$.
Define $f: \mathbb{Z}/6\mathbb{Z} \to \mathbb{Z}/6\mathbb{Z}$ by
\[
f(x)=
\begin{cases}
x + 1 \ (\mathrm{mod}\ 6) & \text{if } x\in\{0,1,2\},\\
x + 2 \ (\mathrm{mod}\ 6) & \text{if } x\in\{3,4,5\}.
\end{cases}
\]
Then $0 \equiv 3 \pmod{3}$ but $f(0)=1$ and $f(3)=5$ satisfy $1 \not\equiv 5 \pmod{3}$.
Hence $f$ is not congruence-preserving modulo $3$, and Theorem~\ref{thm:dcrt} does not apply: the functional graph $G_f$
cannot be decomposed as a product of graphs over $\mathbb{Z}/2\mathbb{Z}$ and $\mathbb{Z}/3\mathbb{Z}$ in the sense of
Definition~\ref{def:graph_product}.

Concretely, the functional graph $G_f$ contains the $4$-cycle
\[
1 \longrightarrow 2 \longrightarrow 3 \longrightarrow 5 \longrightarrow 1,
\]
with transient vertices $0 \to 1$ and $4 \to 0$.
This illustrates that congruence-preservation is a necessary hypothesis for the product decomposition.
\end{example}

\begin{remark}[Scope of the DCRT]
\label{rem:dcrt_scope}
The class of congruence-preserving maps is broad: it includes all polynomially induced maps (Lemma~\ref{lem:poly_cp}) and many rule-based systems used in practice. Quantifying the prevalence of this class within $\End(R_m)$ (e.g., as a fraction of all endomorphisms or in terms of natural measures) is left for future work.
\end{remark}

\begin{lemma}[Cylinders as reduction fibers]
\label{lem:cylinders_as_fibers}
Let $\cO$ be a complete discrete valuation ring with uniformizer $p$, and fix $k \ge 1$. Let $R = \cO/p^k\cO$ and let $\pi^{(k)}: \cO \to R$ be the quotient map (reduction at depth $k$). For each $x \in R$, define the cylinder (closed ball) $\mathcal{B}_x := (\pi^{(k)})^{-1}(x)$. Let $\phi: \cO \to \cO$ be any map and $F: R \to R$ any map. Then the following are equivalent:
\begin{enumerate}
    \item \textbf{Geometric inclusion:} $\phi(\mathcal{B}_x) \subseteq \mathcal{B}_{F(x)}$ for all $x \in R$.
    \item \textbf{Algebraic commutativity:} $\pi^{(k)} \circ \phi = F \circ \pi^{(k)}$ as maps $\cO \to R$.
\end{enumerate}
\end{lemma}
\begin{proof}
(2) $\Rightarrow$ (1): If $z \in \mathcal{B}_x$, then $\pi^{(k)}(z) = x$. By (2), $\pi^{(k)}(\phi(z)) = F(\pi^{(k)}(z)) = F(x)$, so $\phi(z) \in (\pi^{(k)})^{-1}(F(x)) = \mathcal{B}_{F(x)}$.

(1) $\Rightarrow$ (2): Fix $a \in \cO$ and set $x = \pi^{(k)}(a)$. Then $a \in \mathcal{B}_x$, so by (1), $\phi(a) \in \mathcal{B}_{F(x)}$, i.e.\ $\pi^{(k)}(\phi(a)) = F(x) = F(\pi^{(k)}(a))$.
\end{proof}
The Witt-ring formulation of this equivalence is Lemma~\ref{lem:strong_commutation} (stated earlier in this section), when $\cO \simeq W(\kappa)$ and cylinders are the truncation fibers $\mathbb{B}_N(x)$.

\medskip
By Lemma~\ref{lem:cylinders_as_fibers}, interpretation with inclusion (ball images contained in target cylinders) is equivalent to commutativity of the reduction diagram whenever cylinders are defined as the fibers of a depth-$k$ quotient map. The following proposition formalizes the construction of a global interpreter for composite alphabets as a product of local interpreters.

\begin{proposition}[CRT product realization]
\label{prop:crt_product}
Let $m = \prod_{i=1}^s p_i^{k_i}$ be the prime factorization of the modulus. Set
$R_m = \mathbb{Z}/m\mathbb{Z}$ and $R_i = \mathbb{Z}/p_i^{k_i}\mathbb{Z}$. Let
$\Theta: R_m \xrightarrow{\sim} \prod_{i=1}^s R_i$ be the Chinese Remainder isomorphism.
Assume $F: R_m \to R_m$ is congruence-preserving in the sense of Theorem~\ref{thm:dcrt}, so that it
decomposes componentwise into maps $F_i: R_i \to R_i$ via
\[
F \;=\; \Theta^{-1} \circ (F_1 \times \cdots \times F_s) \circ \Theta .
\]
For each $i$, take $K_i = \mathbb{Q}_{p_i}$ and $\cO_i = \mathbb{Z}_{p_i}$ (so the ramification index $e_i=1$); then the quotient map $\pi_i^{(k_i)}:\mathbb{Z}_{p_i}\to R_i=\mathbb{Z}/p_i^{k_i}\mathbb{Z}$ induces the identification $\mathbb{Z}_{p_i}/p_i^{k_i}\mathbb{Z}_{p_i}\cong R_i$. (The ramified case is beyond the scope here; see Remark~\ref{rem:dcrt_ramified_scope}.)
Let $\pi_i^{(k_i)}: \mathbb{Z}_{p_i} \to R_i$ denote the depth-$k_i$ reduction map. (No analytic hypothesis is needed for the commutative square; continuity/analyticity matters only for the non-Archimedean dynamics developed elsewhere.) Suppose there exist maps (rational or analytic) $\phi_i:\mathbb{Z}_{p_i}\to\mathbb{Z}_{p_i}$
such that
\[
\pi_i^{(k_i)} \circ \phi_i \;=\; F_i \circ \pi_i^{(k_i)}
\quad\text{on }\mathbb{Z}_{p_i}.
\]
Define the \emph{product phase space} $\boldsymbol{\Omega} := \prod_{i=1}^s \mathbb{Z}_{p_i}$, the product map
$\Phi := \prod_{i=1}^s \phi_i : \boldsymbol{\Omega} \to \boldsymbol{\Omega}$, and the induced global reduction
\[
\Pi := \Theta^{-1} \circ (\pi_1^{(k_1)} \times \cdots \times \pi_s^{(k_s)}): \boldsymbol{\Omega} \to R_m .
\]
Then $\Pi \circ \Phi = F \circ \Pi$; that is, $\Phi$ interprets $F$ with inclusion on $\boldsymbol{\Omega}$.
\end{proposition}
\begin{proof}
Let $\boldsymbol{z} = (z_1, \ldots, z_s) \in \boldsymbol{\Omega}$. Then
\[
\Pi(\Phi(\boldsymbol{z})) = \Theta^{-1}\bigl( \pi_1^{(k_1)}(\phi_1(z_1)), \ldots, \pi_s^{(k_s)}(\phi_s(z_s)) \bigr).
\]
By the local commutativity $\pi_i^{(k_i)} \circ \phi_i = F_i \circ \pi_i^{(k_i)}$, this equals $\Theta^{-1}(F_1(\pi_1^{(k_1)}(z_1)), \ldots, F_s(\pi_s^{(k_s)}(z_s)))$. Since $F = \Theta^{-1} \circ (\prod_i F_i) \circ \Theta$, the right-hand side is $F(\Pi(\boldsymbol{z}))$.
\end{proof}

\begin{remark}[No single local field/ring for composite modulus]
\label{rem:no_single_local_field_for_composite}
If $m$ has at least two distinct primes, then $\mathbb{Z}/m\mathbb{Z}$ is not a local ring, whereas every quotient $\mathcal{O}/\mathfrak{m}^k$ of a local ring $(\mathcal{O},\mathfrak{m})$ is local. Hence $\mathbb{Z}/m\mathbb{Z}$ cannot be of the form $\mathcal{O}_K/\mathfrak{m}_K^k$. For the product decomposition of the discrete dynamics (DCRT), the corresponding phase space is $\prod_i \mathbb{Z}_{p_i}$ (Proposition~\ref{prop:crt_product}).
\end{remark}

\begin{remark}[Ramified extensions: out of scope]
\label{rem:dcrt_ramified_scope}
If $K_i/\mathbb{Q}_{p_i}$ is ramified ($e_i>1$), then $p_i$ is not a uniformizer: the maximal ideal is $\mathfrak{m}_i=(\pi_i)$ with $|\pi_i|=p_i^{-1/e_i}$, and the quotients $\mathcal{O}_{K_i}/\mathfrak{m}_i^{k}$ are not canonically $\mathbb{Z}/p_i^{k}\mathbb{Z}$. In particular, $|\mathcal{O}_{K_i}/\mathfrak{m}_i^{k}|=|\kappa_i|^{k}=p_i^{f_i k}$, which already differs from $|\mathbb{Z}/p_i^{k}\mathbb{Z}|=p_i^{k}$ when $f_i>1$; even when $f_i=1$ the identification is non-canonical and the relevant reduction is modulo $\mathfrak{m}_i^{k}$ rather than $p_i^{k}$. Theorem~\ref{thm:dcrt} is formulated over $\mathbb{Z}/m\mathbb{Z}$ via the classical CRT and relies on its decomposition; extending to ramified local quotients would require a different formulation (e.g.\ appropriate local quotients and a modified decomposition hypothesis). That is beyond the scope of this paper.
\end{remark}

\subsection{Two orthogonal decompositions: horizontal vs.\ vertical}
\label{subsec:horizontal_vertical}
Our constructions separate two independent sources of complexity for finite dynamics.

\textbf{Horizontal (CRT) factorization} concerns splitting a composite alphabet size $m = \prod_i p_i^{k_i}$ into prime-power components. Proposition~\ref{prop:crt_product} shows that, at the level of interpreters with inclusion (Lemma~\ref{lem:cylinders_as_fibers}), the ambient analytic object for this decomposition is the product $\prod_i \mathbb{Z}_{p_i}$, and the global discrete map is induced componentwise. The examples in Section~\ref{subsec:dcrt_examples} illustrate this at the discrete level.

\textbf{Vertical (depth/degree) refinement} concerns increasing resolution within a fixed prime. Passing from $p$ to $p^k$ corresponds to deeper congruence cylinders defined by $\pi^{(k)}$ (Lemma~\ref{lem:cylinders_as_fibers}), while replacing $\F_p$ by $\F_{p^f}$ corresponds to unramified degree-$f$ lifts modeled by Witt vectors $W(\F_{p^f})$ (Section~\ref{sec:composite_alphabets}; \cite{HazewinkelWitt, SerreLocalFields}). Exact periodic cycles lift along the tower under a non-degeneracy condition (Theorem~\ref{thm:hensel_exact_cycle}, Corollary~\ref{cor:hensel_tower}); see Corollary~\ref{rem:prop_tower} and Appendix~\ref{app:profinite_addendum}. These two axes are orthogonal: CRT governs prime splitting; Witt theory and depth-$k$ reduction govern $p$-adic precision.

\subsection{\texorpdfstring{Pro-$p$}{Pro-p} towers and illustrative examples}
\label{subsec:prop_towers_hensel}
Refinement along a fixed prime via $(\mathbb{Z}/p^n\mathbb{Z})_{n\ge 1}$ yields compatible towers and their limit in $\Zp$. The link between towers and 1-Lipschitz dynamics is summarized in the following corollary; for the lifting of \emph{exact} periodic cycles (Hensel) we refer to Section~\ref{subsec:hensel_cycles} (Theorem~\ref{thm:hensel_exact_cycle}, Corollary~\ref{cor:hensel_tower}, Remarks~\ref{rem:period_divides}--\ref{rem:hensel_degenerate}).

\begin{corollary}[Compatible towers $\Leftrightarrow$ 1-Lipschitz on $\mathbb{Z}_p$]
\label{rem:prop_tower}
Let $R_n = \mathbb{Z}/p^n\mathbb{Z}$ and $\pi_{n+1,n}: R_{n+1} \to R_n$ the reduction map. A family of maps $f_n: R_n \to R_n$ is \emph{compatible} if and only if $\pi_{n+1,n} \circ f_{n+1} = f_n \circ \pi_{n+1,n}$ for all $n \ge 1$. By the universal property of $\mathbb{Z}_p \cong \varprojlim R_n$, a compatible tower induces a unique map $f: \mathbb{Z}_p \to \mathbb{Z}_p$ with $f(x) \bmod p^n = f_n(x \bmod p^n)$ for all $n$; in the $p$-adic setting $f$ is 1-Lipschitz, and conversely every 1-Lipschitz $f$ gives a compatible tower (see~\cite{AnashinKhrennikov, BenedettoDynamics}). Extending a given $f_n$ to a compatible $f_{n+1}$ is always possible but highly non-unique.
\end{corollary}

Figure~\ref{fig:xplus1-tower} illustrates the counterexample $f(x)=x+1$: the period doubles at each level ($2$-cycle, $4$-cycle, $8$-cycle) although the reduction maps are surjective.

\begin{figure}[htbp]
\centering
\includegraphics[width=\linewidth,height=0.20\textheight,keepaspectratio]{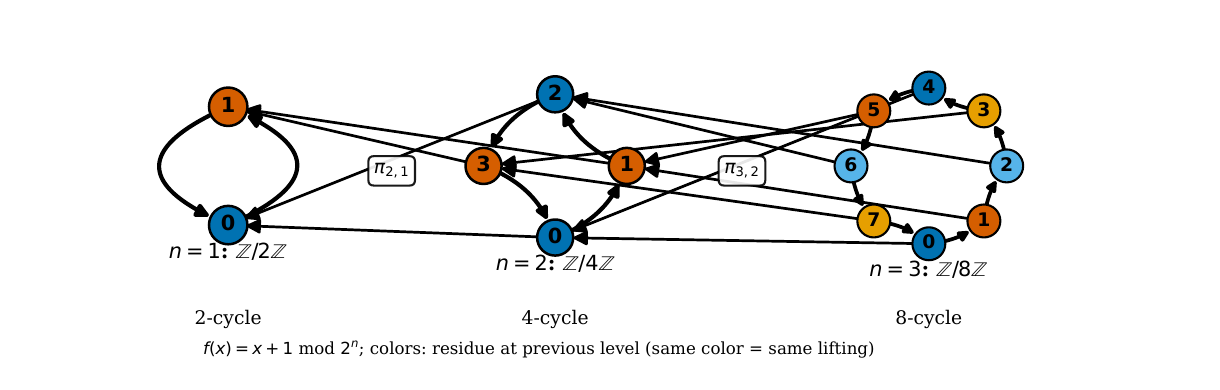}
\caption{Counterexample $f(x)=x+1$ mod $2^n$ for $n=1,2,3$: period grows (2-cycle, 4-cycle, 8-cycle). Each edge is a single arrow; the 2-cycle is drawn as two curved arcs with the same curvature so they do not overlap. The multiplier $\lambda=f'(x)=1$ (parabolic), so Hensel's non-degeneracy fails and cycle length is not preserved. Vertical: reduction (down).}
\label{fig:xplus1-tower}
\end{figure}

A concrete illustration with $p=2$ and $f(x)=x^2+1$ is given in Figure~\ref{fig:pro2-tower}. The polynomial $f(x)=x^2+1$ induces a compatible tower $f_n: \mathbb{Z}/2^n\mathbb{Z} \to \mathbb{Z}/2^n\mathbb{Z}$ automatically (each $f_n$ is the reduction of the same polynomial), so the tower is the discrete shadow of the map $\phi(z)=z^2+1$ on $\mathbb{Z}_2$. The figure shows four levels ($n=1,2,3,4$): each panel is the functional graph at that modulus, with arrows $x \mapsto f(x)$. Vertex colours are consistent across levels: the same colour indicates the same residue class (so reduction goes from level $n$ to $n-1$ by matching colours; lifting is the reverse). Reading the tower \emph{upward}, this encodes the \textbf{liftings} of the discrete dynamics to the unique 1-Lipschitz map $f$ on $\mathbb{Z}_2$ that induces the tower (Corollary~\ref{rem:prop_tower}). Cycles and transients appear at each level; under the non-degeneracy condition of Theorem~\ref{thm:hensel_exact_cycle}, periodic points (e.g.\ the 2-cycle visible at low levels) lift coherently to $\mathbb{Z}_2$ via Hensel.

\begin{figure}[htbp]
\centering
\includegraphics[width=\linewidth,height=0.26\textheight,keepaspectratio]{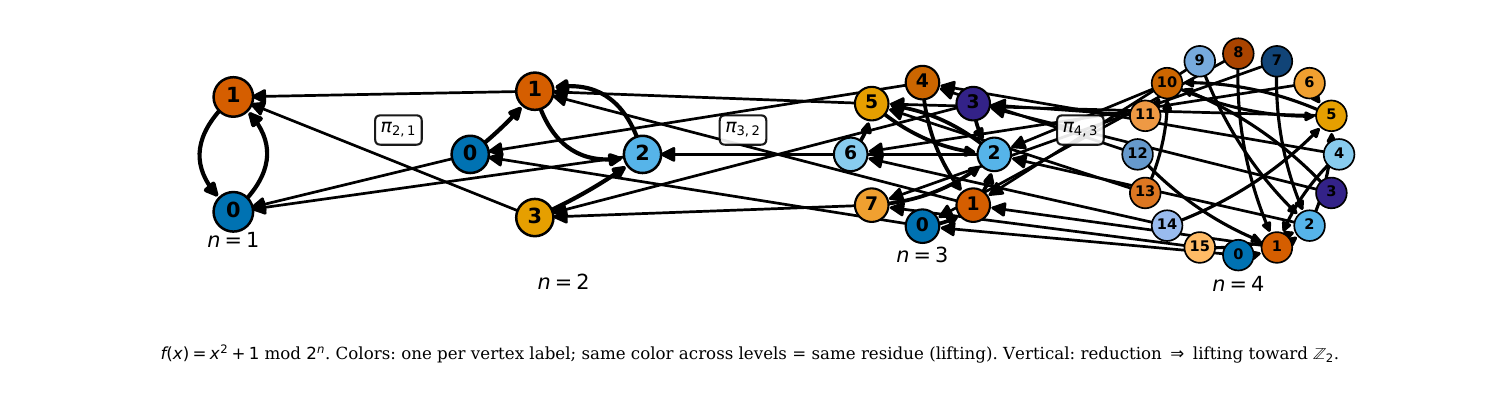}
\caption{Four compatible levels of the functional graphs of $f(x)=x^2+1$ modulo $2^n$ for $n=1,2,3,4$. Each panel shows the digraph at that modulus (arrows: $x\mapsto f(x)$). \textbf{Colours:} one colour per vertex label; the same colour in different panels indicates the same residue class (lifting: a vertex at level $n$ projects to the vertex of the same colour at level $n-1$). Reading the tower \emph{upward}, this encodes lifting toward the limit map on $\mathbb{Z}_2$; commutativity $\pi_{n+1,n}\circ f_{n+1}=f_n\circ\pi_{n+1,n}$ holds. Under the non-degeneracy condition of Theorem~\ref{thm:hensel_exact_cycle}, periodic points lift to $\mathbb{Z}_2$.}
\label{fig:pro2-tower}
\end{figure}

The same picture extends to a general \emph{profinite} (or Stone) setting. Let $X = \varprojlim X_n$ be an inverse limit of finite sets $X_n$ with surjective bonding maps $\pi_{n+1,n}: X_{n+1} \to X_n$. A family of maps $f_n: X_n \to X_n$ is \emph{compatible} if $\pi_{n+1,n} \circ f_{n+1} = f_n \circ \pi_{n+1,n}$ for all $n$; then there is a unique map $f: X \to X$ such that the restriction to ``level $n$'' coincides with $f_n$ (i.e.\ the projective system of dynamical systems $(X_n, f_n)$ lifts to a dynamics $f$ on the limit $X$). In $\mathbb{Z}_p$, each level $n$ is encoded by a partition into balls of radius $p^{-n}$, and compatibility becomes geometric nesting of balls. When a \emph{single} rational map $\phi$ (e.g.\ on $\mathbb{Z}_p$ or $\cO_K$) interprets $f_n$ at each finite level $n$ on the corresponding cylinders (Lemma~\ref{lem:strong_commutation}), $\phi$ induces the limit map $f$ and one says that $\phi$ \emph{interprets $f$ on the limit}. If instead one has only a tower of maps $\phi_n$ (each interpreting the finite dynamics at level $n$), passing to a limiting interpreter requires a \emph{compatibility selection principle}; producing such compatible $\phi_n$ is an interesting direction and is not guaranteed by the Runge/$\varepsilon$-gluing existence at each level. Appendix~\ref{app:profinite_addendum} formalizes the passage from ball-level to pointwise (radius-zero) interpretation.

\subsection{Reduction Theory and the Reduction-Logic Link}
\label{subsec:reduction}

\paragraph{Convention (good reduction = strict).}
Here, \textbf{``good reduction''} always means \emph{strict} good reduction (degree-preserving reduction). See Remark~\ref{rem:strict_good_reduction_convention} and Example~\ref{ex:non_strict_good_reduction}; in parts of the literature the terminology differs.

The use of Witt vectors allows us to connect our interpreters to the theory of \emph{reduction} of rational maps~\cite{BenedettoReduction}. This provides a formal bridge between the continuous dynamics on $\cO_K$ and the discrete dynamics on the residue field $\kappa$.

Recall that a rational map $\phi \in K(z)$ has \emph{good reduction} (here: strict, i.e.\ degree-preserving) if it extends to a morphism $\phi: \PP^1_{\cO_K} \to \PP^1_{\cO_K}$ of the projective line over $\cO_K$, which induces a well-defined map $\overline{\phi}: \PP^1_{\kappa} \to \PP^1_{\kappa}$ on the residue field. This notion is standard in arithmetic dynamics~\cite{BenedettoReduction, SilvermanAD} and arithmetic geometry~\cite{PoonenRCA}; in the sense of~\cite{BenedettoReduction}, good reduction implies uniform non-expanding behavior on the integral locus, and we use it here as a selection principle.

\begin{lemma}[Reduction-Logic Link for Depth 1]
\label{lem:reduction_link}
Let $\phi \in K(z)$ be a rational map with \emph{strict} good reduction (i.e., the reduction $\overline{\phi}$ has the same degree as $\phi$), and let $F: \kappa \to \kappa$ be a discrete rule. Extend $F$ to $\mathbb{P}^1(\kappa)$ by setting $F(\infty) = \infty$. Then $\phi$ interprets $F$ with inclusion at depth $N=1$ (i.e., on the residue field $\kappa$) if and only if the reduction $\overline{\phi}$ coincides with $F$ as maps on $\PP^1(\kappa)$.
\end{lemma}
\begin{proof}
By Lemma~\ref{lem:strong_commutation}, $\phi$ interprets $F$ with inclusion at depth 1 if and only if $\pi_1 \circ \phi = F \circ \pi_1$ on $\cO_K$, where $\pi_1: \cO_K \to \kappa$ is the reduction map modulo $p$. Since $\phi$ has strict good reduction, it preserves the valuation ring $\cO_K$ and induces $\overline{\phi}$ on $\kappa$ with the same degree. The condition $\pi_1 \circ \phi = F \circ \pi_1$ is equivalent to $\overline{\phi} = F$ on $\kappa$ (the affine line). Since $\overline{\phi}$ is a rational map on $\PP^1(\kappa)$ and $F(\infty) = \infty$ by definition, the coincidence extends to all of $\PP^1(\kappa)$.
\end{proof}

\begin{remark}[Convention on Good Reduction]
\label{rem:strict_good_reduction_convention}
We use ``good reduction'' to mean \emph{strict good reduction} (degree-preserving reduction), consistent with Benedetto~\cite{BenedettoReduction}. In parts of the literature, ``good reduction'' is used for any morphism extension (even when the degree drops on the special fiber); we follow the stricter convention so that the equivalence $\overline{\phi} = F$ in Lemma~\ref{lem:reduction_link} holds and the discrete shadow coincides with $F$ and captures the combinatorial structure. For maps whose reduction has lower degree, interpretation with inclusion fails (Example~\ref{ex:non_strict_good_reduction}).
\end{remark}

\begin{example}[Non-Strict Good Reduction (in Some Conventions) Fails Interpretation with Inclusion]
\label{ex:non_strict_good_reduction}
In part of the literature, a rational map $\phi$ is said to have \emph{good reduction} if it extends to a morphism on $\PP^1_{\cO_K}$, even when the reduced map has degree strictly less than $\deg(\phi)$. Under \emph{our} convention (good reduction = strict), such a map does not have ``good reduction.'' For instance, let $K = \Qp$ with $p \geq 3$ and $\phi(z) = z^2/p$. Then $\phi$ extends to a morphism, but the reduction $\overline{\phi}(z) = 0$ has degree 0, so $\phi$ does not have strict good reduction.

Now consider a discrete rule $F: \F_p \to \F_p$ defined by $F(0) = 0$ and $F(x) = x^2$ for $x \neq 0$. The functional graph $G_F$ has a fixed point at $0$ and a 2-cycle for quadratic residues. However, since $\overline{\phi}(z) = 0$ is constant, it maps all of $\PP^1(\F_p)$ to a single point, collapsing the entire functional graph $G_F$ to a single vertex. This violates the interpretation-with-inclusion condition: for any $x \in \F_p$ with $F(x) \neq 0$, we have $\overline{\phi}(x) = 0 \neq F(x)$, so $\phi$ cannot interpret $F$ with inclusion at depth 1.

This example illustrates why strict good reduction (degree preservation) is necessary for interpretation with inclusion: without it, the reduced map can collapse multiple discrete states into a single point, destroying the combinatorial structure of the functional graph.
\end{example}

\begin{lemma}[Good Reduction Induces Compatible Tower]
\label{lem:good_reduction_tower}
Let $\phi \in K(z)$ be a rational map with good reduction. Then $\phi$ induces a compatible tower of maps $\{\phi_N: W_N(\kappa) \to W_N(\kappa)\}_{N \ge 1}$ such that for all $N \ge 1$, the reduction $\overline{\phi}_N$ is well-defined and $\phi$ interprets $\overline{\phi}_N$ with inclusion at depth $N$.
\end{lemma}
\begin{proof}
Since $\phi$ has good reduction, it extends to a morphism $\phi: \PP^1_{\cO_K} \to \PP^1_{\cO_K}$ of the projective line over $\cO_K$. For each $N \ge 1$, this induces a well-defined map $\phi_N: \PP^1(\cO_K/p^N\cO_K) \to \PP^1(\cO_K/p^N\cO_K)$ via the truncation/quotient map $\pi_N:\cO_K\to \cO_K/p^N\cO_K$ and the standard identification $\cO_K/p^N\cO_K \cong W_N(\kappa)$ (when $K$ is unramified). The compatibility follows from the functoriality of reduction modulo $p^N$: for $M \le N$, the truncation map $\tau_{N,M}: W_N(\kappa) \to W_M(\kappa)$ commutes with $\phi_N$ and $\phi_M$ because $\phi$ is defined globally on $\PP^1_{\cO_K}$. By Lemma~\ref{lem:strong_commutation}, $\phi$ interprets $\overline{\phi}_N$ with inclusion at depth $N$ for all $N \ge 1$.
\end{proof}

\begin{observation}[Reduction as Selection Principle]
\label{obs:reduction_link}
Lemma~\ref{lem:reduction_link} establishes that for maps with good reduction, the ``discrete shadow'' at the residue field level coincides with $F$ (Lemma~\ref{lem:reduction_link}), not merely an approximation. Lemma~\ref{lem:good_reduction_tower} shows that good reduction automatically produces a compatible tower, connecting the reduction theory with the profinite setting of Section~\ref{subsec:towers-profinite}. The arithmetic invariants of such maps, such as their Néron--Ogg--Shafarevich criteria and arboreal representations~\cite{PerezBuendia2025Arboreal}, provide invariants that can guide the selection of optimal interpreters (see Section~\ref{subsec:reduction-selection}).
\end{observation}
 \section{Closing the Loop: Analytical Synthesis for Discrete Systems}
\label{sec:closing_the_loop}

\begin{sloppypar}
This section returns from the construction of $p$-adic interpreters to the original combinatorial system. The $p$-adic analytical setting provides a rigorous way to classify discrete states and to select ``better'' models based on stability.
\end{sloppypar}

\subsection{Analytical Invariants vs. Combinatorial States}

In a purely combinatorial functional graph $G_F$, a fixed point $x$ (with $F(x)=x$) is represented by the self-loop $(x,x)$ (convention of Section~\ref{sec:encoding}). However, an interpreter $\phi$ lifts this state to a ball $B_x$ where $\phi(B_x) \subseteq B_x$. The $p$-adic multiplier $\lambda = \phi'(\alpha)$ at a fixed point $\alpha \in B_x$ provides a hierarchy of stability that the graph $G_F$ cannot detect. The classification of multipliers is standard in $p$-adic dynamics~\cite{BenedettoDynamics, RiveraLetelier2003, SilvermanAD}. In the specific setting of ball-based interpreters for \emph{finite} functional graphs, the use of multipliers as analytical signatures for discrete states provides a quantitative refinement of the combinatorial structure:

\begin{enumerate}[leftmargin=2em]
    \item \textbf{Contractive fixed points ($|\lambda| < 1$, also called ``attracting'' in the classical dynamics literature):} These represent ``robust'' discrete states. If the system is perturbed by a noise $\varepsilon$ satisfying the robust exactness bounds (Theorem~\ref{thm:robust_exact}), the state $x$ remains an attractor. In the context of regulatory networks, this corresponds to a stable biological phenotype that persists under continuous variations.
    \item \textbf{Expansive fixed points ($|\lambda| > 1$, also called ``repelling''):} These represent ``sensitive'' or ``transient'' states. Even if $x$ is a fixed point in the discrete model, any continuous infinitesimal perturbation will push the trajectory out of the cylinder $B_x$. Combinatorially, these states might be part of an attractor cycle, but analytically they are unstable.
    \item \textbf{Indifferent fixed points ($|\lambda| = 1$):} These represent ``marginal'' stability, where the local behavior is an isometry.
\end{enumerate}

By fixing an interpreter $\phi$ and reading its local multipliers (equivalently, the ratios $\sigma_i$ of image radii to target radii as in Remark~\ref{rem:image_ball}), one obtains a well-defined local stability type at each state cylinder: contractive, indifferent, or expansive. In this way the same discrete functional graph can be enriched by analytic invariants coming from the $p$-adic dynamics of $\phi$, connecting the combinatorial model with classical non-Archimedean notions (Fatou/Julia decomposition and Siegel disks); see~\cite{Hsiao2013, RiveraLetelier2003}.

\begin{observation}[Runge Approximation and Good Reduction]
\label{obs:runge_good_reduction}
The rational interpreters constructed via Runge approximation (Theorem~\ref{thm:runge}) do not automatically have good reduction. In general, a rational function $\phi$ obtained by approximating a piecewise affine model $\psi$ may have poles or critical points that migrate into the state cylinders, violating the good reduction condition. To ensure good reduction, one must either:
\begin{enumerate}
    \item Restrict the approximation to a subspace of rational functions with coefficients in $\cO_K$ and unit leading coefficients (ensuring the reduced map has the same degree), or
    \item Apply a post-processing step to adjust $\phi$ so that it extends to a morphism on $\PP^1_{\cO_K}$ with degree-preserving reduction.
\end{enumerate}
The selection principle based on good reduction (Section~\ref{subsec:reduction-selection}) thus requires an additional constraint beyond the Runge approximation itself.
\end{observation}

\subsection{Optimal Interpreters and \texorpdfstring{$\varepsilon$}{epsilon}-Precision}
\begin{sloppypar}
The existence results above are non-unique: many rational maps $\phi$ may interpret the same discrete rule. One is thus led to a selection problem: given a target analytic model $\psi\in\mathcal{O}(U)$ on a finite union of balls $U=\bigsqcup_i B_i$, one may seek a rational $\phi$ that approximates $\psi$ while preserving the intended discrete resolution.
\end{sloppypar}

A typical approach is to apply Runge-type approximation (Theorem~\ref{thm:runge}) to produce $\phi$ with $\|\phi-\psi\|_U<\varepsilon$. The parameter $\varepsilon$ controls which combinatorial features are stable under approximation:
\begin{itemize}
    \item If $\varepsilon$ is chosen small relative to the target radii (in particular $\varepsilon<\min_i t_{\tau(i)}$, and within the regime of Theorem~\ref{thm:robust_exact}), then approximation preserves the intended inclusion/exactness constraints at the level of balls, so that $\phi$ realizes the same transition rule on the prescribed grid.
    \item If $\varepsilon$ is larger, then some cylinders may become expansive relative to their targets (cf. the $\sigma_i$-criterion in Remark~\ref{rem:image_ball}), and the induced behavior may no longer match the original discrete resolution.
\end{itemize}

Accordingly, one may consider trade-offs between approximation accuracy, analytic constraints (e.g. pole-freeness on neighborhoods of $U$), and structural complexity of $\phi$ (such as $\deg(\phi)$), without claiming effectivity or a constructive optimization procedure.

\begin{remark}[Open problem: $\varepsilon$ vs.\ degree]
Given a piecewise affine interpreter $\psi$ on a finite union of balls $U$ and $\varepsilon > 0$, Theorem~\ref{thm:runge} guarantees the existence of a rational $\phi$ with $\|\phi - \psi\|_U < \varepsilon$, but does not provide a degree bound. Whether there exists $\phi$ with $\|\phi - \psi\|_U < \varepsilon$ and $\deg(\phi)$ bounded by a function of the number of balls and $\varepsilon$ (or a counterexample showing that no such bound exists) is open.
\end{remark}

\subsection{Good Reduction as a Selection Heuristic (Discussion)}
\label{subsec:reduction-selection}

\begin{sloppypar}
The results proved here (including Theorem~\ref{thm:robust_exact} and the reduction-logic statements at depth $N=1$) are established within the present manuscript. References~\cite{PBNC2025, PerezBuendia2025Arboreal} are cited only for complementary viewpoints (e.g. alternative constructions and arboreal/Galois characterizations) and for outlook.
\end{sloppypar}

\begin{sloppypar}
The reduction-logic link proven here is at \emph{depth $N=1$} (residue field): Lemmas~\ref{lem:reduction_link} and~\ref{lem:good_reduction_tower} characterize when a rational map with good reduction interprets a discrete rule with inclusion on the residue field and induces a compatible tower. The proofs are given in Section~\ref{subsec:reduction} (Lemmas~\ref{lem:reduction_link} and~\ref{lem:good_reduction_tower}). Deeper arithmetic characterizations (Galois/arboreal) are developed in~\cite{PerezBuendia2025Arboreal}. Building on that baseline, we discuss good reduction as a selection heuristic for choosing optimal interpreters.
\end{sloppypar}

Let $K/\Qp$ be a non-archimedean local field with ring of integers $\cO_K$, and let $\phi\in K(z)$ be a rational map of degree $\ge 2$. Recall that $\phi$ has \emph{strict good reduction} if it admits a representation with coefficients in $\cO_K$ whose reduction $\widetilde{\phi}\in \kappa(z)$ has the same degree (equivalently, $\phi$ extends to a morphism on $\PP^1_{\cO_K}$ of the same degree on the special fiber). Good reduction is well-known in the literature to preclude the appearance of new wild ramification in dynamical towers and to enforce a uniform non-expanding behavior on the integral locus~\cite{BenedettoReduction, BenedettoDynamics}; we use this here as context, not as a new claim. The theory of good reduction for morphisms of varieties is developed systematically in~\cite{PoonenRCA}, where it is shown that good reduction ensures well-behaved specialization properties.

\begin{proposition}[Baseline Selection Criterion]
\label{prop:baseline_selection}
Let $\phi \in K(z)$ be a rational map with \emph{strict} good reduction, and let $F: \kappa \to \kappa$ be a discrete rule. Then:
\begin{enumerate}
    \item $\phi(\mathcal{O}_K) \subseteq \mathcal{O}_K$ (i.e., $\phi$ preserves the valuation ring).
    \item $\phi$ interprets $F$ with inclusion at depth $N=1$ if and only if the reduction $\overline{\phi}$ coincides with $F$ as maps on $\mathbb{P}^1(\kappa)$ (by Lemma~\ref{lem:reduction_link}).
    \item $\phi$ induces a compatible tower $\{\phi_N: W_N(\kappa) \to W_N(\kappa)\}_{N \ge 1}$ (by Lemma~\ref{lem:good_reduction_tower}).
\end{enumerate}
In particular, by Lemma~\ref{lem:reduction_link}, the discrete shadow at the residue field level coincides with $F$, not merely an approximation.
\end{proposition}
\begin{proof}
(1) This is the definition of good reduction: $\phi$ extends to a morphism on $\PP^1_{\cO_K}$, hence $\phi(\mathcal{O}_K) \subseteq \mathcal{O}_K$.

(2) This is exactly Lemma~\ref{lem:reduction_link}.

(3) This is exactly Lemma~\ref{lem:good_reduction_tower}.
\end{proof}

\begin{remark}[Arboreal Characterization---Outlook]
\label{rem:arboreal_outlook}
Beyond the baseline established in Proposition~\ref{prop:baseline_selection}, good reduction admits a deeper \emph{Galois/arboreal characterization} along a natural open residual locus: in~\cite{PerezBuendia2025Arboreal} it is shown (in the tame setting) that strict good reduction is equivalent to the triviality of inertia on the \emph{orbital colimit} $\mathcal{G}_{O^{+}(x)}$, a well-defined Galois object encoding the dynamics along a forward orbit. For all integral basepoints lying in a nonempty residual open set (the complement of the reduced postcritical set), this condition is equivalent to the reduced fiber polynomials having unit leading coefficient and unit discriminant. In particular, for such basepoints, all backward preimage extensions are unramified at every level. 

This deeper characterization, while not proven here, provides additional motivation for using good reduction as a selection criterion: it ensures that the induced residue-level dynamics and associated preimage towers remain stable under refinement and do not generate spurious ramification.
\end{remark}

\begin{sloppypar}
Within the space $\mathrm{Int}_\star(F)$ of interpreters of a fixed finite discrete dynamics $F$, we propose \emph{good reduction type} as a selection heuristic, to be combined with geometric invariants (degree, height) and dynamical stratifications (counts of contractive/expansive/indifferent cycles in prescribed balls). Proposition~\ref{prop:baseline_selection} provides the theoretical foundation for this heuristic in the setting of this paper.
\end{sloppypar}

\subsection{The Necessity of a Moduli Space}

The construction provided by Theorem~\ref{thm:runge} proves existence but does not provide a classification. The space of all rational interpreters $\mathrm{Int}_\star(F)$ is an infinite-dimensional object. To make the selection of models systematic, we need a way to organize this space.

Our "Robust Exactness" result (Theorem~\ref{thm:robust_exact}) shows that the set of exact interpreters is an \emph{open} subset of the space of analytic maps. This leads to the following structural result for the space of interpreters (under the fixed-center normalization):

\begin{proposition}[Openness of Multiplier and Dominance Constraints]
\label{prop:openness_multiplier_dominance}
Fix a ball $B(a,r) \subset K$ and let $\cO(B(a,r))$ be endowed with the sup norm $\norm{h}_{B(a,r)} = \sup_{z \in B(a,r)} |h(z)|$. Let $u \in K^\times$ and let $\psi(z) = b + u(z-a)$ be affine on $B(a,r)$. Then there exists $\delta > 0$ such that for every $g \in \cO(B(a,r))$ with $\norm{g - \psi}_{B(a,r)} < \delta$, the following hold:
\begin{enumerate}
\item $|g'(a)| = |u|$;
\item writing the Taylor expansion $g(z) = \sum_{k \ge 0} c_k (z-a)^k$ on $B(a,r)$, one has
\[
\max_{k \ge 2} |c_k| r^{k-1} < |g'(a)|,
\]
hence $g$ has \emph{linear dominance} on $B(a,r)$ and therefore $g(B(a,r))$ is exactly a ball of radius $|g'(a)| r$.
\end{enumerate}
In particular, the constraints $|g'(a)| = |u|$ and linear dominance define an \emph{open} condition in the sup-norm topology on $\cO(B(a,r))$.
\end{proposition}
\begin{proof}
Let $h = g - \psi$. By the non-Archimedean Cauchy estimate on the closed ball $B(a,r)$ (Lemma~\ref{lem:cauchy}; in the dynamical setting see~\cite{BenedettoDynamics}).\footnote{The Archimedean analogue is the Cauchy inequality for holomorphic functions on a disk.} If $h(z) = \sum_{k \ge 0} d_k (z-a)^k$ then
\[
|d_k| \leq \frac{\norm{h}_{B(a,r)}}{r^k} \qquad (k \ge 0).
\]
Choose $\delta > 0$ such that $\delta < |u| r$. If $\norm{h}_{B(a,r)} < \delta$, then
\[
|g'(a) - u| = |d_1| \leq \frac{\norm{h}_{B(a,r)}}{r} < |u|.
\]
By ultrametricity, $|g'(a)| = |u|$, proving (1). For $k \ge 2$ we have
\[
|c_k| = |d_k| \leq \frac{\norm{h}_{B(a,r)}}{r^k} < \frac{|u| r}{r^k} = |u| r^{1-k},
\]
hence $|c_k| r^{k-1} < |u| = |g'(a)|$ for all $k \ge 2$, proving (2). The last claim (image is exactly a ball) follows from Lemma~\ref{lem:dominance_image}.
\end{proof}

\begin{remark}[Value Group Warning]
\label{rem:value_group_warning}
Do \emph{not} invoke discreteness of the value group unless $K/\mathbb{Q}_p$ is finite (and even then it is $p^{\frac{1}{e}\mathbb{Z}}$). For $K = \mathbb{C}_p$ the value group is $p^{\mathbb{Q}}$ and is dense in $\mathbb{R}_{>0}$. The openness argument above uses only ultrametricity, hence works uniformly.
\end{remark}

\begin{corollary}[Stratification of Interpreter Space]
\label{cor:stratification}
For any finite functional graph $G_F$, the space of exact rational interpreters $\mathrm{Int}_{\mathrm{exact}}(F)$ satisfying $|\phi(a_i) - b_{\tau(i)}| \leq t_{\tau(i)}$ is a union of open strata $\mathcal{S}_{\vec{\lambda}}$, where $\vec{\lambda} = (|\lambda_1|, \dots, |\lambda_k|)$ is the vector of multiplier magnitudes at the fixed points. Each stratum $\mathcal{S}_{\vec{\lambda}}$ is stable under $\varepsilon$-perturbations.
\end{corollary}
\begin{proof}
This follows from Proposition~\ref{prop:openness_multiplier_dominance} applied to each ball $B_i$ with $u_i$ from the piecewise affine model. The stratification follows by partitioning the space of exact interpreters according to these invariant vectors.
\end{proof}

\subsection{Hierarchical Equivalence and Moduli Spaces}

To classify interpreters, we must account for coordinate changes. In the $p$-adic setting, the natural group of transformations consists of \emph{affine isometries} that preserve the ball hierarchy.

\begin{definition}[Rigid Isometric Conjugacy of Interpreters]
\label{def:isometric_conjugacy}
Fix a finite family of balls $\mathcal{B} = \{B_i = B(a_i, r_i)\}_{i=1}^N$ in $K$. Let $\mathrm{Iso}(\mathcal{B})$ be the group of \emph{affine isometries}
\[
\sigma(z) = \alpha z + \beta, \qquad \alpha \in K^\times, \quad |\alpha| = 1, \quad \beta \in K,
\]
such that $\sigma$ permutes the family $\mathcal{B}$ (i.e., $\sigma(B_i) = B_{\pi(i)}$ for some permutation $\pi$). Two interpreters $\phi, \psi$ are \emph{isometrically conjugate relative to $\mathcal{B}$} if there exists $\sigma \in \mathrm{Iso}(\mathcal{B})$ such that
\[
\psi = \sigma^{-1} \circ \phi \circ \sigma.
\]
\end{definition}

\begin{lemma}[Conjugacy Invariance of Ball-Dynamics and Multipliers under Affine Isometries]
\label{lem:conjugacy_invariance_affine_isometries}
Let $\sigma(z) = \alpha z + \beta$ with $|\alpha| = 1$. Then:
\begin{enumerate}
\item For every ball $B(c,r)$ one has $\sigma(B(c,r)) = B(\sigma(c), r)$.
\item If $\psi = \sigma^{-1} \circ \phi \circ \sigma$, then for every $x$ where derivatives exist,
\[
\psi'(x) = \phi'(\sigma(x)).
\]
In particular, $|\psi'(x)| = |\phi'(\sigma(x))|$.
\item If $\phi$ satisfies linear dominance on $B(c,r)$, then $\psi$ satisfies linear dominance on $B(\sigma^{-1}(c), r)$, and the interpretation type at each ball (contractive, indifferent, or expansive) is preserved.
\end{enumerate}
\end{lemma}
\begin{proof}
(1) Immediate from ultrametricity and $|\alpha| = 1$: $|z-c| \leq r \iff |\alpha z + \beta - (\alpha c + \beta)| = |\alpha| |z-c| \leq r$.

(2) By the chain rule for analytic maps and the fact that $\sigma'(x) = \alpha$ is constant:
\[
\psi'(x) = (\sigma^{-1})'(\phi(\sigma(x))) \cdot \phi'(\sigma(x)) \cdot \sigma'(x) = \alpha^{-1} \cdot \phi'(\sigma(x)) \cdot \alpha = \phi'(\sigma(x)).
\]

(3) Combine (1)--(2) with the definition of linear dominance via Taylor coefficients on balls. Because $\sigma$ is an isometry, radii are preserved; because (2) preserves the absolute value of the linear coefficient, the strict inequality $\max_{k \ge 2} |c_k| r^{k-1} < |c_1|$ is invariant under conjugacy by $\sigma$.
\end{proof}

\begin{definition}[Hierarchical Equivalence]
\label{def:hierarchical_equivalence}
Two interpreters $\phi_1, \phi_2 \in \mathrm{Int}_\star(F)$ are \emph{hierarchically equivalent} if they are isometrically conjugate relative to the family of state cylinders $\mathcal{B} = \{B_i\}$.
\end{definition}

\begin{corollary}[Invariance of Semantics]
\label{lem:equivalence_invariance}
Hierarchical equivalence preserves the interpretation type at each ball (contractive, indifferent, or expansive) and the magnitudes of the multipliers $|\lambda_i|$.
\end{corollary}
\begin{proof}
This follows immediately from Lemma~\ref{lem:conjugacy_invariance_affine_isometries}.
\end{proof}

\begin{remark}[Larger Conjugacy Groups]
\label{rem:bigger_group_needs_analyticity}
If you want a larger conjugacy group (e.g., general ball-preserving homeomorphisms), you must \emph{not} use derivatives. To keep derivative-based invariants (multipliers, dominance), restrict to analytic isometries; affine isometries are the safest minimal choice.
\end{remark}

\begin{sloppypar}
The definitions above (isometric conjugacy, hierarchical equivalence) are \emph{preliminary steps} toward a moduli theory of $p$-adic interpreters: they fix the notion of equivalence and the invariants (multipliers, dominance) that should be preserved. A systematic development---e.g., description of $\mathrm{Int}_\star(F)/{\sim}$ as a geometric object, dimension, explicit stratification---is left for future work.
\end{sloppypar}
 \section{Worked Examples and Arithmetic Applications}
\label{sec:examples}

\subsection{A Concrete Case Study: Cycle and Transient Dynamics}
The passage from a functional graph to its $p$-adic interpretation and analytical synthesis is illustrated by a system with four states $X = \{0, 1, 2, 3\}$. The existence and robustness results used here are proved in Sections~5--6; the Dynamic CRT and its factorization are proved in Section~7. This section illustrates them with explicit transition graphs and arithmetic examples.

\subsubsection{Discrete Dynamics}
Define the transition map $F: X \to X$ as follows:
\begin{equation}
\label{eq:ex_discrete_rule}
F(0) = 1, \quad F(1) = 0, \quad F(2) = 1, \quad F(3) = 3.
\end{equation}
The functional graph $G_F$ (see Figure~\ref{fig:functional_graph}) consists of a limit cycle of length 2 $\{0, 1\}$, a transient state $\{2\}$, and a fixed point $\{3\}$.

\begin{figure}[ht]
    \centering
    \includegraphics[width=0.55\textwidth]{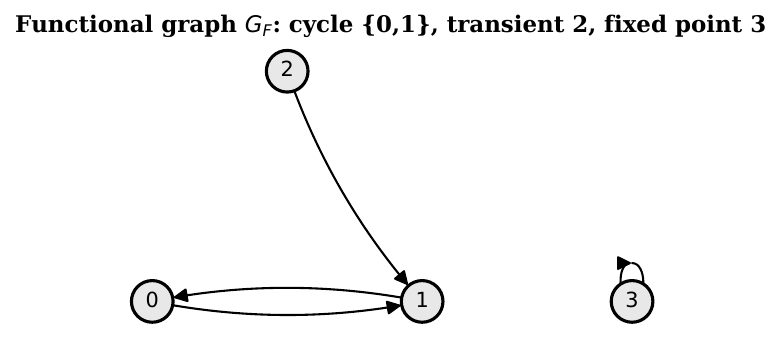}
    \caption{Functional graph $G_F$ with a cycle $\{0,1\}$, a transient state $2$, and a fixed point $3$ (self-loop at $3$, by convention in Section~\ref{sec:encoding}).}
    \label{fig:functional_graph}
\end{figure}

\subsubsection{\texorpdfstring{$p$}{p}-adic Encoding}
We use $p=2$ and $n=2$ levels. The states are mapped to the centers $\{0, 1, 2, 3\} \subset \mathbb{Z}_2$ and associated with cylinders of radius $2^{-2} = 1/4$:
\begin{equation}
\label{eq:ex_cylinders}
B_0 = 4\mathbb{Z}_2, \quad B_1 = 1 + 4\mathbb{Z}_2, \quad B_2 = 2 + 4\mathbb{Z}_2, \quad B_3 = 3 + 4\mathbb{Z}_2.
\end{equation}

\subsubsection{Analytic Interpreter Construction}
Using Theorem~\ref{thm:analytic_existence}, we construct a piecewise affine exact interpreter $\psi$. For simplicity, we choose it to be isometric on each ball ($|u_i| = 1$):
\begin{equation}
\label{eq:ex_psi_def}
\psi(z) = \begin{cases}
1 + z & z \in B_0 \\
(z-1) & z \in B_1 \\
1 + (z-2) = z-1 & z \in B_2 \\
3 + (z-3) = z & z \in B_3
\end{cases}
\end{equation}
Note that for $z \in B_2$, $|\psi(z) - 1| = |z-2| \leq 1/4$, so $\psi(B_2) = B_1$.

\subsubsection{Global Rational Interpreter}
By Theorem~\ref{thm:runge} and Lemma~\ref{lem:runge_finite_constraints}, there exists a rational function $\phi \in \mathbb{Q}_2(z)$ such that $\|\phi-\psi\|_{U}<\varepsilon$ for some $\varepsilon<1/4$ on $U=\bigsqcup_i B_i$. In particular, $\phi$ satisfies the same finite-level constraints on the centers and realizes the intended transitions on the prescribed cylinders (Sections~5--6). The Runge step is non-effective in general with respect to degree/height bounds; for a complementary construction with an explicit formula see~\cite{PBNC2025} and the discussion in Section~\ref{sec:discussion}.

For a minimal example, take the 2-state system $X = \{0, 1\}$ and the swap map $F(0) = 1$, $F(1) = 0$. With $p=2$ and $n=1$, the cylinders are $B_0 = 2\mathbb{Z}_2$ and $B_1 = 1 + 2\mathbb{Z}_2$ (radius $1/2$). The piecewise affine interpreter $\psi(z) = 1 - z$ satisfies $\psi(0) = 1$ and $\psi(1) = 0$. In this case $\psi$ is already a polynomial, so we may take $\phi(z) = 1 - z$ directly. Verification: if $z \in B_0$ then $z$ is even, so $1 - z$ is odd and $1 - z \in B_1$; if $z \in B_1$ then $z$ is odd, so $1 - z$ is even and $1 - z \in B_0$. Thus $\phi(B_0) = B_1$ and $\phi(B_1) = B_0$ (exact interpretation). Moreover $\phi'(z) = -1$, so $|\phi'(z)|_2 = 1$ (indifferent/isometric everywhere). For larger systems (e.g.\ the 4-state example above), explicit formulas with simple coefficients are not always available and Runge approximation (Lemma~\ref{lem:runge_finite_constraints}) is required; the existence of some rational interpreter is guaranteed by Theorem~\ref{thm:runge}.

\subsubsection{Robustness and Multipliers}
By Theorem~\ref{thm:robust_exact}, any rational function $\phi$ that is $1/4$-close to $\psi$ and satisfies $|\phi(a_i) - b_{F(i)}| \leq 1/4$ will be an exact interpreter, i.e., $\phi(B_i) = B_{F(i)}$ for all $i$. 
Suppose we select an interpreter $\phi$ such that its fixed point in $B_3$ is $\alpha = 3$. The multiplier $\lambda = \phi'(3)$ determines the stability. If $|\lambda|_2 < 1$, then $B_3$ is an \emph{analytical sink}. This illustrates how the $p$-adic lifting adds a "metric depth" to the combinatorial graph.

\subsection{Dynamics on Witt Vectors: The case of \texorpdfstring{$GF(4)$}{GF(4)}}
Consider a system with 4 states forming the finite field $\mathbb{F}_4$. We use the ring of Witt vectors $W(\mathbb{F}_4)$, which is the ring of integers $\mathcal{O}_K$ of the unramified extension of $\mathbb{Q}_2$ of degree 2. The Teichmüller representatives $\Teich(\mathbb{F}_4) = \{0, 1, \omega, \omega^2\}$ where $\omega^2 + \omega + 1 = 0$ in $\mathcal{O}_K$ provide the centers for our state cylinders.

A discrete map $F: \mathbb{F}_4 \to \mathbb{F}_4$ is interpreted by a 2-adic map $\phi$ on $\mathcal{O}_K$. For example, the Frobenius-type map $F(\xi) = \xi^2$ is exactly interpreted by $\phi(z) = z^2$. 
\begin{itemize}
    \item \textbf{Good Reduction:} The map $\phi(z) = z^2$ has good reduction modulo 2, as its reduction $\overline{\phi}(\xi) = \xi^2$ is a well-defined map of degree 2 on $\mathbb{F}_4$.
    \item \textbf{Stability:} The 2-cycle $\{\omega, \omega^2\}$ is contractive because the multiplier of the second iterate satisfies:
\begin{equation}
\label{eq:gf4-stability}
|(\phi^2)'(\omega)|_2 = |4 \omega^3|_2 = |4|_2 = 1/4 < 1.
\end{equation}
Note that $|\omega|_2 = 1$ since $\omega$ is a Teichmüller lift (a root of unity in the unramified extension).
Furthermore, at depth $N \ge 2$, the action on cylinders is exact since linear dominance is satisfied on balls of radius $r \le 1/4$.
\end{itemize}
This example shows that by requiring our interpreters to have \textbf{good reduction} in the sense of Benedetto~\cite{BenedettoReduction}, we ensure that the continuous dynamics "collapses" perfectly to the discrete logic $F$, while the $p$-adic multipliers provide a rigorous measure of the system's robustness to continuous perturbations.

\subsection{Arithmetic Case Studies: Frobenius and Verschiebung}
\label{subsec:arithmetic-examples}

The Witt vector setting provides two examples of interpreters that are purely arithmetic in nature, yet illustrate the semantics perfectly.

\subsubsection{Example A: Frobenius as an Exact Interpreter}
Let $K$ be an unramified extension of $\Qp$ and let $\sigma \in \Gal(K/\Qp)$ be the Frobenius automorphism. By Serre~\cite{SerreLocalFields}, $\sigma$ is an isometry of $\cO_K$ and induces the $p$-power map $\xi \mapsto \xi^p$ on the residue field.
If we define the discrete rule $F$ as the $p$-power Frobenius on the alphabet $\kappa$, then $\sigma$ satisfies:
\[
\pi_N \circ \sigma = F \circ \pi_N
\]
for every depth $N$. Thus, $\sigma$ is an \textbf{interpreter with inclusion} for all resolution levels. Since $\sigma$ is an isometry, it preserves the radius of balls, making it an \textbf{exact interpreter}: $\sigma(\mathbb{B}_N(x)) = \mathbb{B}_N(F(x))$. This is the cleanest example of a continuous lifting that exactly tracks a discrete rule.

\subsubsection{Example B: Verschiebung as a Shift-Register Interpreter}
Let $V: W_N(\kappa) \to W_{N+1}(\kappa)$ be the Verschiebung operator. In the ring of Witt vectors, the identity $p \cdot \mathbb{1} = V \circ F$ holds, where $F$ is the Frobenius map. Let $\phi(z) = p z$ be the multiplication by $p$ on $\cO_K$. Then the induced action on the truncated vectors satisfies:
\begin{equation}
\label{eq:verschiebung-identity}
\pi_{N+1}(p z) = V(F(\pi_N(z))).
\end{equation}
In terms of digit expansions, $\phi$ shifts the sequence of states and applies the Frobenius automorphism to the existing states. This provides an exact interpretation of a \emph{twisted} shift-register update rule.

\subsection{Modular Squaring and Analytical Stability}
Consider the modular squaring map $F(x) = x^2 \pmod{p^n}$ on $\mathbb{Z}/p^n\mathbb{Z}$. Lifted to $\phi(z) = z^2$ on cylinders $B_x$, the analytical interpreter reveals a stability structure invisible to the graph $G_F$:
\begin{itemize}[leftmargin=2em]
    \item \textbf{The contractive fixed point at $0$:} For any $p$, $z=0$ is superattracting ($\lambda=0$). The cylinder $B_0$ is a contractive basin.
    \item \textbf{The $p=2$ vs $p>2$ Dichotomy:} At $x=1$, the multiplier is $\lambda = 2$. If $p=2$, $|2|_2 = 1/2 < 1$ (contractive); if $p>2$, $|2|_p = 1$ (indifferent/isometric).
\end{itemize}

\subsection{Parameter Spaces: The \texorpdfstring{$p$}{p}-adic \texorpdfstring{$z^2+c$}{z squared plus c} family}
Consider the family $\phi_c(z) = z^2 + c$ with $c \in \mathbb{Z}_p$. The following result shows a fundamental rigidity in the discrete structure.

\begin{proposition}[Combinatorial Rigidity]
\label{prop:rigidity}
Let $c_1, c_2 \in \mathbb{Z}_p$. If $c_1 \equiv c_2 \pmod{p^n}$, then the functional graphs induced by $\phi_{c_1}$ and $\phi_{c_2}$ on the cylinders $\{B_x\}_{x \in \mathbb{Z}/p^n\mathbb{Z}}$ are identical.
\end{proposition}
\begin{proof}
The transition between cylinders $B_x \to B_y$ is determined by the condition $\phi(x) \equiv y \pmod{p^n}$. For any $x \in \mathbb{Z}_p$, we have $\phi_{c_1}(x) = x^2 + c_1 \equiv x^2 + c_2 = \phi_{c_2}(x) \pmod{p^n}$. Thus, $x \bmod p^n$ maps to the same residue class under both maps.
\end{proof}

However, the \emph{internal} dynamics can bifurcate. For $p=3, n=1$, if $c=0$, $z=0$ is superattracting. If $c=3$, the fixed point $\alpha$ satisfies $\alpha^2 - \alpha + 3 = 0$; its multiplier $|\lambda|_3 = |2\alpha|_3 = 1/3$ shows that $c=3$ provides a different analytic profile.

\subsection{Algebraic Systems: DCRT and PDS}
\label{subsec:dcrt_examples}
So far we have built interpreters for systems on residue rings $\mathbb{Z}/p^n\mathbb{Z}$ or, via Witt vectors, on $\mathbb{Z}/m\mathbb{Z}$ for a single modulus $m$ (Sections 2--6, 8). When $m$ factors into coprime moduli, the Dynamic CRT (Section~\ref{subsec:dcrt}) decomposes the \emph{discrete} dynamics horizontally: the global functional graph factors as a product of local graphs. This is the combinatorial shadow of the ring-theoretic Chinese Remainder Theorem under dynamic hypotheses: $\text{CRT (rings)}\quad \leadsto \quad \text{DCRT (functional graphs)}$. The DCRT can be read in two directions: \emph{reduction} (global $\to$ local) and \emph{lifting} (local $\to$ global). The two examples below illustrate each.

\subsubsection{DCRT ``going'' (reduction)}
Start from a global map $F(x) = x^2 + 1$ on $\mathbb{Z}/6\mathbb{Z}$. Since $F$ is congruence-preserving (it is induced by a polynomial), Theorem~\ref{thm:dcrt} implies that $G_F \cong G_{F_2} \times G_{F_3}$, where $F_2$ and $F_3$ are the maps induced by $F$ on $\mathbb{Z}/2\mathbb{Z}$ and $\mathbb{Z}/3\mathbb{Z}$ by reduction modulo $2$ and modulo $3$. Figure~\ref{fig:dcrt_graphs} illustrates this \emph{reduction}: the graph over $\mathbb{Z}/6\mathbb{Z}$ (top) is decomposed into the two local graphs (bottom). Each vertex $x \in \mathbb{Z}/6\mathbb{Z}$ reduces to $(x \bmod 2, x \bmod 3)$; the bicolor coding (left half = residue mod $2$, right half = residue mod $3$) makes this visible, and each edge in $G_F$ projects to the corresponding edge in $G_{F_2}$ and in $G_{F_3}$.
\begin{figure}[ht]
    \centering
    \includegraphics[width=\textwidth]{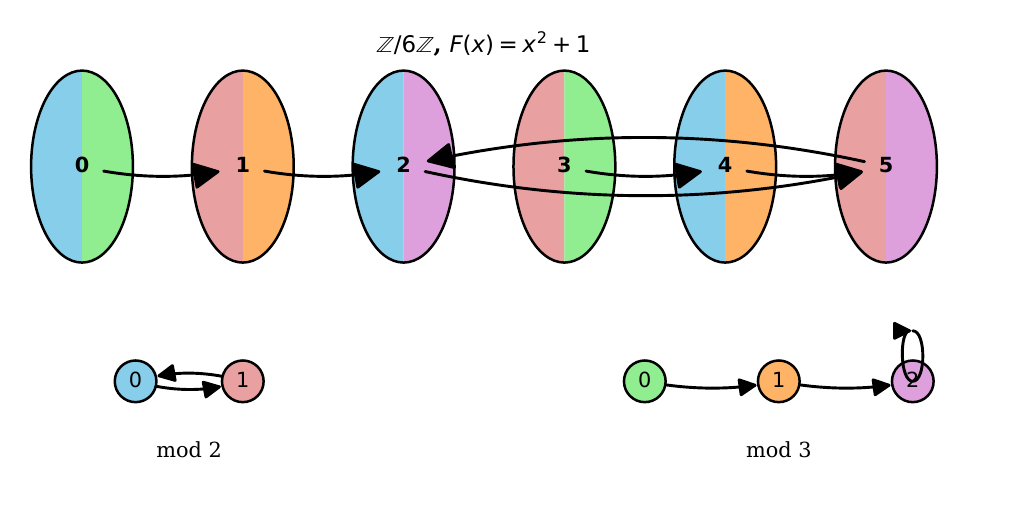}
    \caption{DCRT ``going'' (reduction): the global graph $G_F$ on $\mathbb{Z}/6\mathbb{Z}$ (top) reduces to the local graphs $G_{F_2}$ (mod $2$) and $G_{F_3}$ (mod $3$) (bottom). Vertices are bicolored by residues modulo $2$ (left half) and modulo $3$ (right half), making the reduction $x\mapsto(x\bmod 2,\;x\bmod 3)$ visible.}
    \label{fig:dcrt_graphs}
\end{figure}

\subsubsection{DCRT ``return'' (lifting)}
Conversely, start from two \emph{arbitrary} local graphs: one on $\mathbb{Z}/3\mathbb{Z}$ and one on $\mathbb{Z}/4\mathbb{Z}$. The classical CRT gives $\mathbb{Z}/12\mathbb{Z} \cong \mathbb{Z}/3\mathbb{Z} \times \mathbb{Z}/4\mathbb{Z}$ via $\Theta$. For any $f_3: \mathbb{Z}/3\mathbb{Z} \to \mathbb{Z}/3\mathbb{Z}$ and $f_4: \mathbb{Z}/4\mathbb{Z} \to \mathbb{Z}/4\mathbb{Z}$, we \emph{lift} them to a single map $f: \mathbb{Z}/12\mathbb{Z} \to \mathbb{Z}/12\mathbb{Z}$ by defining $f(x) = \Theta^{-1}(f_3(x \bmod 3), f_4(x \bmod 4))$: each $x$ is sent to the unique element of $\mathbb{Z}/12\mathbb{Z}$ whose reduction mod $3$ is $f_3(x \bmod 3)$ and whose reduction mod $4$ is $f_4(x \bmod 4)$. Figure~\ref{fig:dcrt_arbitrary} illustrates this \emph{lifting}: the two local graphs (bottom) are assembled into the global graph $G_f$ on $\mathbb{Z}/12\mathbb{Z}$ (top). Vertices in $G_f$ are again bicolor (left = residue mod $3$, right = residue mod $4$), and each edge in $G_f$ is the unique lift of the corresponding pair of edges in the factors.
\begin{figure}[htbp]
    \centering
    \includegraphics[width=\textwidth]{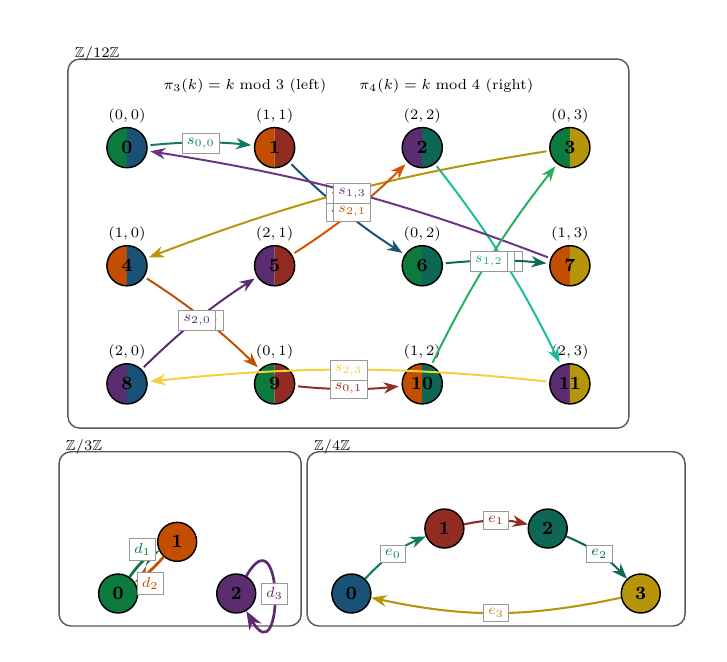}
    \caption{DCRT ``return'' (lifting): the local graphs $G_{f_3}$ (mod $3$) and $G_{f_4}$ (mod $4$) (bottom) are lifted by $\Theta^{-1}$ to the global \emph{functional digraph} $G_f$ on $\mathbb{Z}/12\mathbb{Z}$ (top): each vertex has exactly one outgoing edge $x \to f(x)$, where $f$ is the unique map that reduces to $f_3$ mod $3$ and $f_4$ mod $4$. Vertices are bicolored by residues modulo $3$ and modulo $4$. Edges in $\mathbb{Z}/3\mathbb{Z}$ and $\mathbb{Z}/4\mathbb{Z}$ are labeled $d_i$ and $e_j$ respectively; each edge in the global graph is labeled $s_{i,j}$ for the unique common lift (gluing) of $d_i$ and $e_j$.}
    \label{fig:dcrt_arbitrary}
\end{figure}
Under reduction, fixed points go to fixed points and periodic orbits (cycles) to cycles; under lifting, the same holds. In the product picture, a vertex $(x_1, x_2)$ is a \emph{fixed point} of $f$ if and only if $x_1$ and $x_2$ are fixed by $f_1$ and $f_2$ respectively; it lies on a \emph{cycle} of length $k$ if and only if each $x_i$ lies on a cycle of $f_i$ whose length divides $k$ (the period in the product is the lcm of the local periods). So, for example, a fixed point in one factor and a 2-cycle in the other lift to a 2-cycle in the global graph.

\subsubsection{Interpolation and reduction}
One can recover the polynomial formulation by asking for polynomial maps that realize the given functional graphs. For each of the three rings $R \in \{\mathbb{Z}/3\mathbb{Z}, \mathbb{Z}/4\mathbb{Z}, \mathbb{Z}/12\mathbb{Z}\}$, there exist polynomials $P_R(X) \in \mathbb{Z}[X]$ (obtained, for instance, by Lagrange interpolation on the vertices) such that the induced map $x \mapsto P_R(x) \bmod |R|$ coincides with the prescribed transition. For a congruence-preserving map $f$ on $\mathbb{Z}/12\mathbb{Z}$ that arises from $(f_3, f_4)$ via CRT, a polynomial $P_{12}(X)$ that induces $f$ on $\mathbb{Z}/12\mathbb{Z}$ satisfies $P_{12}(x) \equiv f_3(x) \pmod{3}$ and $P_{12}(x) \equiv f_4(x) \pmod{4}$ for all $x \in \mathbb{Z}/12\mathbb{Z}$; thus reducing $P_{12}$ modulo $3$ (respectively modulo $4$) yields a polynomial whose induced map on $\mathbb{Z}/3\mathbb{Z}$ (respectively $\mathbb{Z}/4\mathbb{Z}$) is exactly $f_3$ (respectively $f_4$). This illustrates how the DCRT decomposition is reflected at the level of polynomial interpreters: the global interpreter on $\mathbb{Z}/12\mathbb{Z}$ reduces correctly to the local ones.

\subsubsection{Relation to \texorpdfstring{$p$}{p}-adic interpreters}
The examples above are at the \emph{discrete} level: graphs and maps on residue rings. In the setting of this paper, the goal is to realize such discrete dynamics as the shadow of continuous $p$-adic maps (interpreters) on cylinders or balls (Sections 2, 4). We work over a complete non-Archimedean field $K$ (e.g.\ $K = \mathbb{C}_p$); the CRT decomposition is used at the discrete level, while the existence of a $p$-adic interpreter is obtained by constructing a ball system with $m$ balls and applying Runge (Theorem~\ref{thm:runge}, Lemma~\ref{lem:runge_finite_constraints}). \textbf{Reduction} then has an interpreter counterpart: a rational map $\phi$ that interprets a congruence-preserving map $F$ on $\mathbb{Z}/m\mathbb{Z}$ (with $m$ composite) induces, by reduction modulo each prime power dividing $m$, maps that interpret the local factors $F_i$; Lemma~\ref{lem:strong_commutation} and the good-reduction link (Section~\ref{subsec:reduction}) make this precise. In the reduction example with $F(x)=x^2+1$ on $\mathbb{Z}/6\mathbb{Z}$, one can even take $\varphi(z)=z^2+1$ as an interpreter (sketch): over $\mathbb{Z}_p$ ($p=2$ or $3$), take balls $B(a, p^{-n})$; since $\varphi$ has good reduction and is 1-Lipschitz on $\mathbb{Z}_p$, we have $\varphi(B(a,p^{-n})) \subseteq B(\varphi(a), p^{-n})$, so $\varphi$ interprets the reduction of $F$ with inclusion; Lemma~\ref{lem:good_reduction_tower} yields the compatible tower. \textbf{Lifting} corresponds to the converse: when the discrete map $f$ on $\mathbb{Z}/12\mathbb{Z}$ is assembled from $f_3$ and $f_4$ via CRT, the \emph{product} of local $p$-adic interpreters (one per prime factor) on the space $\prod_i \cO_i$ induces $f$ on $\mathbb{Z}/12\mathbb{Z}$ (Proposition~\ref{prop:crt_product}); existence of \emph{some} interpreter for $f$ on a single field (e.g.\ a union of 12 balls in $\mathbb{C}_p$ for a chosen prime $p$) is still guaranteed by Theorem~\ref{thm:runge}. Thus the DCRT examples are not isolated: they show how the discrete decomposition (reduction and lifting) aligns with the way interpreters are constrained by good reduction and by the hierarchical structure of Section~\ref{sec:encoding}. The following diagram distinguishes the \emph{discrete} CRT lift (product of residue rings) from the \emph{$p$-adic} interpreter (maps on balls in $K$):
\[
\begin{tikzcd}[column sep=2.2em, row sep=1.8em]
  \text{(balls in }K) \rar{\varphi} \dar[swap]{\scriptstyle \bmod p_i^{k_i}} & \text{(balls in }K) \dar{\scriptstyle \bmod p_i^{k_i}} \\
  \mathbb{Z}/p_i^{k_i}\mathbb{Z} \rar{F_{p_i^{k_i}}} & \mathbb{Z}/p_i^{k_i}\mathbb{Z}
\end{tikzcd}
\qquad
\text{and}\qquad
\mathbb{Z}/m\mathbb{Z} \cong \prod_i \mathbb{Z}/p_i^{k_i}\mathbb{Z}.
\]
Here DCRT enters as a statement about functional graphs: if $F$ is congruence-preserving, the global graph is isomorphic to the product of the local graphs. The hypothesis is not decorative: Example~\ref{ex:non_cp} (Section~\ref{subsec:dcrt}) shows that when $F$ is not congruence-preserving, the conclusion of the DCRT fails.

The examples above highlight two complementary arithmetic decompositions of discrete dynamics:
(1)~decomposition across coprime moduli via the Chinese remainder theorem, and
(2)~refinement along a fixed prime via the inverse system $(\mathbb{Z}/p^n\mathbb{Z})_{n\ge 1}$.
The second direction leads to profinite (inverse-limit) dynamics on $\Zp$
and to the question of passing from ball-level to pointwise interpretation; see
Appendix~\ref{app:profinite_addendum}.

In \emph{Polynomial Dynamical Systems} (PDS) common in regulatory modeling,\footnote{Such systems appear in modeling cell-fate in organisms like \emph{C. elegans}, where discrete variables represent signaling levels.} selecting an interpreter with good reduction and $|\lambda|<1$ yields a continuous system that preserves the discrete phenotype as a stable attractor.
 \section{Discussion and Future Outlook}
\label{sec:discussion}

\subsection{Structural meaning of the results}
The framework of Sections~2--6 establishes that every finite functional graph admits a rational interpreter on a finite union of balls (Theorem~\ref{thm:runge}, Corollary~\ref{cor:exact_via_runge_interpolation}), and that exact cylinder mapping is stable under perturbation when linear dominance holds (Theorem~\ref{thm:robust_exact}). Sections~7--9 add two orthogonal axes: \emph{horizontal} (CRT decomposition over coprime moduli) and \emph{vertical} (Witt/depth towers). Interpretation with inclusion is equivalent to commutativity with the reduction maps (Lemma~\ref{lem:cylinders_as_fibers}); for composite $m$, the global dynamics on $\prod_i \cO_i$ is the product of the local dynamics (Proposition~\ref{prop:crt_product}). Good reduction then acts as a selection principle among the many possible interpreters (Proposition~\ref{prop:baseline_selection}). Conceptually, the results say that discrete dynamics on residue rings and Witt vectors are realizable as the reduction of rational $p$-adic dynamics, with a well-defined stability structure (contractive/indifferent/expansive) and a clean decomposition when the discrete map is congruence-preserving.

\subsection{Relation to other work}
The present paper proves existence of rational interpreters from purely combinatorial data (finite functional graph or ball system) via Runge approximation (Theorem~\ref{thm:runge}) of a piecewise affine model (Theorem~\ref{thm:analytic_existence}); the proof does not provide constructive or uniform degree/height bounds. In~\cite{PBNC2025}, $\varepsilon$-gluing constructs a global rational map that $\varepsilon$-approximates prescribed local analytic or rational maps on finitely many disjoint balls, under the hypothesis that $f_i(B) \subset B_1(0)$ for all $i$, where $B$ is the union of the balls; the proof of Theorem~4.2 there gives an explicit formula $F_\varepsilon(z)=\sum_i f_i(z)h_i(z)$. Remark~\ref{rem:alternative_methods} summarizes the $\varepsilon$-approximation statement. Theorem~\ref{thm:robust_exact} shows that exact ball mapping is an open condition under linear dominance. Automata theory studies 1-Lipschitz or tower-compatible maps on $\Zp$ that need not be rational~\cite{AnashinKhrennikov, Poonen2014}; here rational (or rigid-analytic) realizations are required so that degree, height, and reduction are defined and Proposition~\ref{prop:baseline_selection} and linear dominance apply. Uniform degree or height bounds are not provided (Remark~\ref{rem:degree_bound_limits}). The Chinese remainder theorem appears in other graph-theoretic or geometric contexts (e.g.\ symmetries of directed graphs~\cite{Foldes1980}, or momentum graphs and restriction maps in equivariant cohomology~\cite{CarrellKaveh2024}); here the graph is the \emph{functional graph} of a single map on $\mathbb{Z}/m\mathbb{Z}$, and the DCRT factorizes the \emph{dynamics} (and hence that graph) under the congruence-preserving hypothesis, in service of $p$-adic interpreter construction.

\subsection{Limitations and open directions}
The following are implied by the results or by the discussion in Appendix~\ref{app:profinite_addendum}.

\begin{enumerate}[label=(\roman*), leftmargin=2em]
\item \textbf{Effective bounds:} Uniform degree/height bounds for rational interpreters in terms of $\varepsilon$ and the number of balls are not provided by the existence proof; the question remains open (Remark~\ref{rem:degree_bound_limits}, Section~\ref{sec:closing_the_loop}).
\item \textbf{Profinite limit:} Appendix~\ref{app:profinite_addendum} proves existence of a continuous (1-Lipschitz) limit and, under a Cauchy condition, of an analytic pointwise interpreter (Proposition~\ref{prop:cauchy_analytic_limit}). Rationality of the limit and existence of compatible rational sequences remain open (Remark~\ref{rem:regularity_limit}).
\item \textbf{Moduli of interpreters:} The space $\Int_\star(F)$ is infinite-dimensional; stratification by multiplier vectors is open (Proposition~\ref{prop:openness_multiplier_dominance}, Corollary~\ref{cor:stratification}). Parameterization and dimension bounds are not established here.
\item \textbf{DCRT scope:} Congruence-preservation is necessary for the product decomposition (Example~\ref{ex:non_cp}); how restrictive this is for typical discrete models remains open (Remark~\ref{rem:dcrt_scope}).
\end{enumerate}

\subsection{Radius-zero interpretation}
The main results are at \emph{finite} resolution: interpreters act on balls of radius $p^{-n}$ and realize a given graph at that level. The passage to pointwise (radius-zero) semantics on $\Zp$ is treated in Appendix~\ref{app:profinite_addendum}: existence of a unique continuous limit and, under additional hypotheses, of an analytic limit is proved there; rationality of the limit is left open. The present manuscript establishes existence, robustness, and structural decomposition of \emph{finite-resolution} rational interpreters (Theorem~\ref{thm:runge}, Theorem~\ref{thm:robust_exact}, Proposition~\ref{prop:crt_product}); the profinite limit and its rationality are discussed in the appendix and remain open where indicated.
 
\appendix
\section{Profinite dynamics and radius-zero interpretation}
\label{app:profinite_addendum}

\subsection{Scope}
\label{app:scope}
The core results above solve an inverse problem at \emph{finite resolution}:
given a finite functional graph encoded by a finite partition of $\Zp$ into residue balls
of radius $p^{-n}$, we construct rational or analytic maps $\phi$ whose action on balls realizes the
prescribed discrete dynamics (Sections~5--6).
This appendix records what is needed to pass from such \emph{ball-level} semantics to a
\emph{pointwise} (``radius-zero'') model for profinite dynamics, i.e.\ dynamics specified by a compatible
tower modulo $p^n$.
Existence of an interpreter $\phi_n$ for each $n$ does \emph{not} by itself imply a limiting interpreter on $\Zp$; see below.

\subsection{Two notions of interpretation}
\label{app:two_notions}
Ball-level interpretation prescribes inclusions or equalities of the form
$\phi(B(a,p^{-n})) \subseteq B(F_n(a),p^{-n})$ for finitely many balls at a fixed radius (Section~\ref{sec:semantics}).
In contrast, \emph{radius-zero interpretation} requires exact pointwise equality on $\Zp$:
\[
\phi(z)=f(z)\quad \text{for all } z\in \Zp.
\]
The latter can be viewed as the limit of ball-semantics as radii $p^{-n}\to 0$.
This is no longer an interpreter in the sense of Definition~\ref{def:interpreter}; it is a \textbf{pointwise realization} of the inverse-limit dynamics.
In the sense of continuous interpolation compatible with the tower, Route~1 (below) gives the unique such map $f$ (the $1$-Lipschitz limit determined by the compatible tower); Route~2 asks when this map can be realized by a rigid-analytic function in $\cO(B(0,1))$.

\subsection{Horizontal versus vertical}
\label{app:horizontal_vertical}
Fix a compatible tower $(X_n,f_n,\pi_{n+1,n})_{n\ge 1}$ with surjective bonding maps and $\pi_{n+1,n}\circ f_{n+1}=f_n\circ \pi_{n+1,n}$. We refer to the dynamics within each level $n$ as \emph{horizontal}, and to the maps $\pi_{n+1,n}$ (reduction downward) and their possible lifts (upward, when available) as \emph{vertical}. Note that vertical lifting is \emph{not} a global morphism in general save under additional hypotheses; it is a lifting of points or cycles (see Theorem~\ref{thm:hensel_exact_cycle} and Remark~\ref{rem:hensel_local_only} in Section~\ref{subsec:hensel_cycles}), not necessarily a section of the inverse system.

\subsection{Profinite dynamics and the unique continuous lift (Route 1)}
\label{app:route1}
Let $R_n=\mathbb{Z}/p^n\mathbb{Z}$ and $\pi_{n+1,n}:R_{n+1}\to R_n$ be reduction.
A family $F_n:R_n\to R_n$ is \emph{compatible} if
$\pi_{n+1,n}\circ F_{n+1}=F_n\circ \pi_{n+1,n}$ for all $n\ge 1$.
By the universal property $\Zp\simeq \varprojlim R_n$, such a tower induces a unique map
$f:\Zp\to \Zp$ satisfying $f(x)\bmod p^n = F_n(x\bmod p^n)$ for all $n$.
Moreover $f$ is $1$-Lipschitz: formally, $x \equiv y \pmod{p^n}$ implies $f(x) \equiv f(y) \pmod{p^n}$ for all $n$ by compatibility, so $|f(x)-f(y)|_p \le p^{-n}$ whenever $|x-y|_p \le p^{-n}$, hence $|f(x)-f(y)|_p \le |x-y|_p$; see also Corollary~\ref{rem:prop_tower}.

Define the locally constant lifts $f^{\mathrm{lc}}_n:\Zp\to\Zp$ by
\[
f^{\mathrm{lc}}_n(z):=\iota_n\!\left(F_n(z\bmod p^n)\right),
\]
where $\iota_n:R_n\hookrightarrow \Zp$ is the standard embedding with image
$\{0,1,\dots,p^n-1\}$.
Then $f^{\mathrm{lc}}_n$ is constant on each residue ball of radius $p^{-n}$, and compatibility yields
\[
\|f^{\mathrm{lc}}_{n+1}-f^{\mathrm{lc}}_n\|_{\infty,\Zp}\le p^{-n}.
\]
Hence $(f^{\mathrm{lc}}_n)$ converges uniformly to $f$.
This produces a unique \emph{continuous} (indeed $1$-Lipschitz) pointwise model of the tower.
Uniform limits of locally constant functions are continuous, but typically \emph{nowhere analytic}; so Route~1 does not yield an interpreter in the analytic/rational sense of the core of the paper.

\subsection{When does one obtain an analytic pointwise interpreter? (Route 2)}
\label{app:route2}
Let $K$ be a complete non-Archimedean field containing $\Qp$ and let
$\cO(B(0,1))$ denote the Banach $K$-algebra of rigid analytic functions on the closed unit ball, with sup norm $\|\cdot\|_{\infty}$ (see, e.g., \cite{BGR} or \cite{BenedettoDynamics}).\footnote{The Archimedean analogue is the algebra of holomorphic functions on the closed unit disk, complete in the sup norm.}
Runge-type or $\varepsilon$-gluing results (e.g., Theorem~\ref{thm:runge}, \cite{PBNC2025}) give \emph{existence at each level}, but not necessarily compatible across levels. So the obstacle is not ``approximating'' but ``approximating \emph{compatibly}.'' The genuinely hard part is \emph{producing} a sequence $(\phi_n)$ inside this algebra that is both level-$n$ correct and Cauchy in sup norm; compatibility across levels is extra structure and is left as an open direction unless a selection principle is proved.

\begin{proposition}[Sufficient Cauchy criterion for analytic radius-zero interpretation]
\label{prop:cauchy_analytic_limit}
Assume a compatible tower $(F_n)$ on $R_n=\mathbb{Z}/p^n\mathbb{Z}$, with induced $1$-Lipschitz
limit map $f:\Zp\to \Zp$.
Suppose there exists a sequence $\phi_n\in \cO(B(0,1))$ such that:
\begin{enumerate}
\item (\emph{Level-$n$ correctness}) for all $z\in \Zp$, one has
$\phi_n(z)\equiv f(z)\pmod{p^n}$ (equivalently, $\phi_n$ induces $F_n$ on $R_n$);
\item (\emph{Uniform Cauchy control}) $\|\phi_{n+1}-\phi_n\|_{\infty,\Zp}\le C\,p^{-n}$
for some constant $C$ independent of $n$.
\end{enumerate}
Then $(\phi_n)$ converges uniformly on $\Zp$ to a limit $\phi\in \cO(B(0,1))$,
and $\phi$ interprets $f$ pointwise on $\Zp$, i.e.\ $\phi(z)=f(z)$ for all $z\in \Zp$.
\end{proposition}
\begin{remark}[Why condition (2) is needed]
Condition (1) alone gives $\|\phi_{n+1}-\phi_n\|_{\infty,\Zp}\le p^{-n}$ (since both agree with $f$ modulo $p^n$ on $\Zp$). For the limit to lie in $\cO(B(0,1))$, the sequence must be Cauchy in the sup norm over the \emph{full} closed unit ball; when (2) is taken with the norm over $B(0,1)$, it provides this and yields $\phi\in \cO(B(0,1))$.
\end{remark}
\begin{proof}
The uniform Cauchy bound implies that $(\phi_n)$ is Cauchy in the sup norm on $\Zp$. By completeness of the rigid-analytic algebra $\cO(B(0,1))$ (see~\cite{BGR}; in the dynamical setting~\cite{BenedettoDynamics})\footnote{In the Archimedean case, the same conclusion holds for holomorphic functions on the disk.} the sequence converges to a limit $\phi\in \cO(B(0,1))$. Level-$n$ correctness gives $\phi(z)\equiv f(z)\pmod{p^n}$ for all $n$, hence $\phi(z)=f(z)$ on $\Zp$.
\end{proof}

\paragraph{Remark on rationality.}
Even if each $\phi_n$ can be chosen rational (by Runge-type constructions on finite unions of balls),
uniform convergence alone does not force the limit to remain rational unless one imposes
\emph{uniform complexity bounds} (e.g.\ bounded degree and controlled poles in a fixed compact set).
Thus, a genuine ``rational radius-zero lifting theorem'' requires extra structure and is left as an
open direction.

\paragraph{Periods, Hensel, and rigidity.}
For an analytic $\phi$, lifting a point of exact period $m$ is governed by Hensel applied to
$g_m(x)=\phi^{\circ m}(x)-x$.
Solving $g_m(x)=0$ only ensures $\operatorname{per}(x)\mid m$; rigidity of the exact period (and lifting an
entire $m$-cycle without collapsing or splitting) typically requires the non-degeneracy condition
$(\phi^{\circ m})'(a)\not\equiv 1\pmod p$.
In the parabolic case $(\phi^{\circ m})'(a)\equiv 1\pmod p$, cycle lengths may grow along the tower
(e.g.\ $\phi(x)=x+1$ on $\mathbb{Z}_2$); see Corollary~\ref{rem:prop_tower}.

\begin{remark}[Regularity and rationality of the limit]
\label{rem:regularity_limit}
Formally: the construction yields a continuous pointwise interpreter; analyticity holds under the Cauchy criterion of Proposition~\ref{prop:cauchy_analytic_limit}; rationality is not asserted and remains an open problem requiring uniform complexity bounds.
\end{remark}

In summary, this appendix formalizes the passage from ball-level to pointwise (radius-zero) interpretation: Route 1 yields a unique continuous ($1$-Lipschitz) limit; Route 2 gives a sufficient Cauchy criterion for an analytic interpreter (Proposition~\ref{prop:cauchy_analytic_limit}). The core existence and robustness results of the paper (Sections~5--6) concern \emph{finite} functional graphs.
 
\section*{Acknowledgments}
The author acknowledges prior joint work with Victor Nopal-Coello and \'Angela Fuquen during their postdoctoral and visiting stays at CIMAT M\'erida (Victor as a postdoctoral researcher, \'Angela as a visitor and postdoc of the project cited in the Funding section below), which the author hosted and which motivated part of this work.

\section*{Statements and Declarations}

\subsection*{Funding}
This work was partially supported by the SECIHTI Fronteras de la Ciencia project ``Modelos matem\'aticos y computacionales no convencionales para el estudio y an\'alisis de problemas relevantes en Biolog\'ia'' (Grant CF 2019/217367).

\subsection*{Competing interests}
The author declares no competing interests.

\subsection*{Data availability}
Data sharing is not applicable to this article as no datasets were generated or analysed during the current study.

\subsection*{Author contributions}
The author conceived the work, carried out the analysis, and wrote the manuscript.

\bibliography{references}

\end{document}